\documentclass[10pt]{article}
\pdfoutput=1
\usepackage{geometry, amsmath, amsthm, latexsym, amssymb, graphicx, mathtools, tikz, enumerate, dsfont}
\usepackage{tikz-cd}
\usepackage[utf8]{inputenc}
\usepackage[OT2,T1]{fontenc}
\usepackage[shortlabels]{enumitem}
\setenumerate{topsep=0pt}

\usepackage{tkz-graph}
\usepackage{graphicx}
\usepackage{csquotes}
\newcommand{\git}{\mathbin{
  \mathchoice{/\mkern-6mu/}
    {/\mkern-6mu/}
    {/\mkern-5mu/}
    {/\mkern-5mu/}}}
\usepackage{lmodern}
\usepackage{arydshln}
\usepackage[shortlabels]{enumitem}
\usepackage{multirow}
\usepackage{mathdots}
\usepackage{hyperref}
\usepackage{authblk}
\usepackage{arydshln}
\usepackage{bigints}

\DeclareSymbolFont{cyrletters}{OT2}{wncyr}{m}{n}
\DeclareMathSymbol{\Sha}{\mathalpha}{cyrletters}{"58}

\theoremstyle{definition}
\newtheorem{defi}{Definition}[section]

\newtheorem{obs}[defi]{Remark}

\theoremstyle{plain}
\newtheorem{theorem}[defi]{Theorem}
\newtheorem{prop}[defi]{Proposition}
\newtheorem{lemma}[defi]{Lemma}
\newtheorem{cor}[defi]{Corollary}
\newtheorem*{idea*}{Idea}

\DeclareMathOperator{\Gal}{Gal}
\DeclareMathOperator{\SL}{SL}
\DeclareMathOperator{\GL}{GL}

\DeclareMathOperator{\So}{SO}
\DeclareMathOperator{\PSo}{PSO}
\DeclareMathOperator{\Po}{PO}
\DeclareMathOperator{\Sp}{Sp}
\DeclareMathOperator{\PSp}{PSp}

\DeclareMathOperator{\id}{id}

\DeclareMathOperator{\ord}{ord}

\DeclareMathOperator{\im}{Im}

\DeclareMathOperator{\Jac}{Jac}

\DeclareMathOperator{\Ad}{Ad}
\DeclareMathOperator{\Sym}{Sym}
\DeclareMathOperator{\Spec}{Spec}
\DeclareMathOperator{\Ht}{ht}

\DeclareMathOperator{\diag}{diag}
\DeclareMathOperator{\Stab}{Stab}

\DeclareMathOperator{\vol}{vol}
\DeclareMathOperator{\Sel}{Sel}
\DeclareMathOperator{\Res}{Res}

\DeclareMathOperator{\Mat}{Mat}

\newcommand{\lp}{\left(}
\newcommand{\rp}{\right)}

\newcommand{\la}{\langle}
\newcommand{\ra}{\rangle}
\newcommand{\Z}{\mathbb{Z}}
\newcommand{\Q}{\mathbb{Q}}
\newcommand{\R}{\mathbb{R}}

\newcommand{\F}{\mathbb{F}}

\newcommand{\hh}{\mathfrak{h}}

\newcommand{\fc}{\mathfrak{c}}

\newcommand{\cC}{\mathcal{C}}
\newcommand{\cA}{\mathcal{A}}

\newcommand{\cL}{\mathcal{L}}

\newcommand{\cS}{\mathcal{S}}
\newcommand{\cF}{\mathcal{F}}
\newcommand{\cJ}{\mathcal{J}}
\newcommand{\cB}{\mathcal{B}}
\newcommand{\cU}{\mathcal{U}}

\newcommand{\cT}{\mathcal{T}}
\newcommand{\cM}{\mathcal{M}}
\newcommand{\eps}{\varepsilon}

\newcommand{\ol}{\overline}

\usepackage[
backend=biber,
style=alphabetic,
sorting=nyt
]{biblatex}
\title{Arithmetic statistics of isogeny Selmer groups associated to hyperelliptic curves}
\author{Martí Oller}
\affil{\small Charles University, Faculty of Mathematics and Physics, Department of Algebra, Sokolovská 83, 186 75 Praha 8, Czech Republic}

\begin{document}

\maketitle
\begin{abstract}
We determine asymptotic results for the average size of Selmer groups
arising from certain isogenies related to Jacobians of
hyperelliptic curves of genus $g\geq 2$. Our results come from two different
routes. The first route is through the geometry-of-numbers methods
pioneered by Bhargava, where we obtain new parametrisations coming from Vinberg theory arising
from representations related to the Dynkin diagrams of type $B$ and $C$.
The second route uses recent results by Koymans--Smith, which relate
the average sizes of Selmer groups to Tamagawa ratios through a formula
by Greenberg--Wiles.
\end{abstract}
\tableofcontents

\section{Introduction}

\subsection{Statement of results}

Let $n \geq 2$. Fix $B = \Spec \Z[p_2,\dots,p_{4n}]$,
and for $b \in B(\R)$ consider the hyperelliptic
curve
\[
C_b \colon y^2 = x(x^{2n}+p_2(b)x^{2n-1}+\dots+p_{4n}(b)).
\]
Let $J_b$ be the Jacobian of $C_b$. We note that $(0,0) \in C_b(\Q)$,
and the difference of this rational point and the point at infinity generates
an order $2$ subgroup of $J_b[2]$, which we denote by $M$. Let $M^{\perp}$ be the orthogonal
complement of $M$ with respect to the Weil pairing. We note that $M \leqslant M^{\perp}$,
and that both $M$ and $M^{\perp}$ are stable under $\Gal(\ol{\Q}/\Q)$-action. Therefore,
there exists an abelian variety $A_b$ with maps
\begin{equation}
\label{eq:JA}
J_b \xrightarrow{\phi_M} A_b \xrightarrow{\phi} A_b^{\vee} \xrightarrow{\phi_M^{\vee}} J_b,
\end{equation}
such that if $\psi = \phi \circ \phi_M$, then $J_b[\phi_M] = M$, $J_b[\psi] = M^{\perp}$,
$A_b[\phi] = M^{\perp}/M$. Moreover, the total composition map $J_b \to J_b$ 
in \eqref{eq:JA} is the multiplication-by-$2$ map.

For an element $b \in B(\R)$, let us define its \emph{height} as
\[
\Ht(b) = \sup_{i=1,\dots,2n}\{|p_{2i}(b)|^{1/(2i)}\}.
\]

For $\phi$, which is a self-dual isogeny, we have the following result.
\begin{theorem}
\label{theo:mainPhi}
We have that
\[
\limsup_{X \to \infty}\frac{\sum_{\Ht(b) < X} \# \Sel_{\phi} A_b }{\sum_{\Ht(b) < X} 1} \leq 6.
\]
\end{theorem}

For some of the other isogenies, we can obtain stronger theorems involving
the $\kappa$-moments of the Selmer groups for any real number $\kappa \geq 0$.

\begin{theorem}
\label{theo:mainMoments}
Let $\kappa \geq 0$. As $X \to \infty$, the following asymptotics hold:
\[
\frac{\sum_{\Ht(b) < X} (\# \Sel_{\phi_M} J_b)^{\kappa} }{\sum_{\Ht(b) < X} 1} \asymp_{\kappa} 1
\]
\[
\frac{\sum_{\Ht(b) < X} (\# \Sel_{\phi_M^{\vee}} A_b^{\vee})^{\kappa} }{\sum_{\Ht(b) < X} 1} \asymp_{\kappa} (\log X)^{2^{\kappa}-1}.
\]
\[
\frac{\sum_{\Ht(b) < X} (\# \Sel_{\psi^{\vee}} A_b)^{\kappa} }{\sum_{\Ht(b) < X} 1} \gg_{\kappa} (\log X)^{2^{\kappa}-1}.
\]
The implied constants depend on $\kappa$.
\end{theorem}

Note that the last point of Theorem \ref{theo:mainMoments} does not give
information about an upper bound for the moments $\#\Sel_{\psi^{\vee}} A_b$.
We do, however, get a result for $\kappa = 1$:

\begin{theorem}
\label{theo:mainPsi}
As $X \to \infty$,
\[
\frac{\sum_{\Ht(b) < X} \# \Sel_{\psi^{\vee}} A_b}{\sum_{\Ht(b) < X} 1} \ll \log X.
\]
\end{theorem}

\begin{obs}
We stress that the following results work for genus at least $2$. For genus
$1$ we have $\phi = 1$, so in particular Theorem \ref{theo:mainPhi} is
trivially true. However, the analogue of Theorems \ref{theo:mainMoments} and
Theorem \ref{theo:mainPsi} is not true. By \cite[Example 7.1]{KS}, we
know that the average size of the Selmer group of a $2$-isogeny in the
given family $y^2 = x(x^2+ax+b)$ is $\asymp \sqrt{\log X}$ (more generally,
the $\kappa$-moment is $(\log X)^{2^{\kappa} + 2^{-\kappa} - 2}$ for any real $\kappa \geq 0$).
\end{obs}

\begin{obs}
We can compare the result of Theorem \ref{theo:mainPhi} with the Poonen--Rains
heuristics in \cite{PoonenRains}. These heuristics contain some predictions
for Selmer groups of self-dual isogenies $\lambda \colon A \to A^{\vee}$
which come from some symmetric line sheaf $\cL$ in $A$. This is the case 
for all of our isogenies $\phi \colon A_b \to A_b^{\vee}$: $\phi$ is 
self-dual by Lemma \ref{lemma:self-dual}, and the obstruction for $\phi$
to come from a symmetric line bundle is measured by an element $c_{\phi} \in H^1(A_b[\phi])$,
which is zero in our case by \cite[Proposition 3.12(f)]{PR2}. Then,
\cite[Theorem 4.14]{PoonenRains} identifies $\Sel_{\phi} J_b$\footnote{The cited \cite[Theorem 4.14]{PoonenRains} only identifies a \emph{quotient} of $\Sel_{\phi}J_b$ as an intersection, but that quotient is equal to $\Sel_{\phi} J_b$ $100\%$ of the time by \cite[Proposition 3.4]{PoonenRains}, using the fact that $A_b[\phi]$ is isomorphic to the $2$-torsion of the Jacobian of $y^2 = f(x)$.}
with an intersection of two maximal isotropic subspaces of an infinite-dimensional
quadratic space over $\F_2$. Then, Theorem \ref{theo:mainPhi} appears
to be consistent with the predictions of the Poonen--Rains heuristics:
the upper bound for our average size coincides with that of $2$-Selmer
groups of even hyperelliptic curves, which in both cases account for the
presence of a marked rational subgroup of the Selmer group.
However, we note that isogenies like $\psi$ or $\phi_M$ are not self-dual,
and thus do not fall within the framework of Poonen--Rains.
\end{obs}

\subsection{Method of proof}

Many results on the average size of Selmer groups of isogenies that are
multiplication-by-$n$ have appeared in the literature in the past years,
helped mainly by Bhargava's striking new methods in geometry-of-numbers:
as seen for instance in \cite{BSquartics, BScubics, LagaThesis}, among many others.
We will use these techniques, as well as results from a recent preprint
from Koymans--Smith \cite{KS}, who use a formula by Greenberg--Wiles 
to obtain statistical information on the size of Selmer groups
in many different settings.

The standard technique in ``Bhargavology'' is to parametrise the elements
of the Selmer groups by integral orbits of a representation $(G,V)$ of
a reductive group $G$ defined over $\Z$. Finding such parametrisations
is one of the main obstacles in obtaining more of these results. Previous
experience suggests that many representations used in arithmetic statistics
actually arise from Vinberg theory, or in other words the study of graded
Lie algebras. In \cite{ThorneThesis}, Thorne connected the Vinberg representations
associated to the $\Z/2\Z$-gradings of the simply laced Lie algebras (i.e.
those of type $A_n$, $D_n$ or $E_n$) with certain families of curves
arising as deformations of simple surface singularities, in such a way that
the orbits of the representation should give arithmetic information about
the constructed families of curves. This perspective has been used, implicitly
and explicitly, to obtain statistical results on the size of $2$-Selmer
groups in the past: all these results have been unified and reproved in
Laga's thesis \cite{LagaThesis}, which gives a uniform proof of all such results.

Other Vinberg representations have appeared in the literature, either coming
from either non-simply laced Dynkin diagrams or higher order gradings
(or both). In \cite{RTE8}, a $\Z/3\Z$-grading in $E_8$ is used to study
the $3$-Selmer group of odd genus $2$ curves. In \cite{3isoEC}, a $\Z/3\Z$-grading of $G_2$ was used to
study $3$-isogeny Selmer groups of the elliptic curves $y^2 = x^3+k$,
a perspective that was later generalised in \cite{3isoAV} for abelian
varieties. In \cite{LagaThesis}, a $\Z/2\Z$-grading in $F_4$ is used to
study $2$-Selmer groups of a family of Prym varieties, in a manner that
serves as a template for our results.

We now explain the structure of this paper. First, in Section \ref{subs:VinbergBC} we introduce the
representations $(G_B,V_B)$ and $(G_C,V_C)$ arising from the stable $2$-grading
of the Dynkin diagrams $B_{2n}$ and $C_{2n}$. Then, in Sections \ref{section:OrbitsB}
and \ref{section:OrbitsC} we will show how the $B$ and $C$-representations
(respectively) are connected to the geometric picture of \eqref{eq:JA}.
First, we observe that the rings of invariants $\Q[V_B]^{G_B}$ and $\Q[V_C]^{G_C}$
are both isomorphic to an affine space $\Q[p_2,\dots,p_{4n}]$, where $p_{2i}$ has degree $2i$.
Therefore, any element of $V_B(\Q)$ or $V_C(\Q)$ can be associated to the
hyperelliptic curve $C_b \colon y^2 = x(x^{2n}+p_2x^{2n-1}+\dots + p_{4n})$.
The most important results of Sections \ref{section:OrbitsB} and \ref{section:OrbitsC}
are the construction of embeddings
\[
\Sel_{\psi} J_b \xhookrightarrow{} \lp G_B(\Q)\backslash V_{B,b}(\Q)\rp \bigcap \frac{1}{2}V_B(\Z), \quad \Sel_{\psi^{\vee}} A_b \xhookrightarrow{} \lp G_C(\Q)\backslash V_{C,b}(\Q)\rp \bigcap \frac{1}{2}V_C(\Z),
\]
where $V_{B,b}$ and $V_{C,b}$ denote the elements of $V_B$ and $V_C$
having invariants $b \in B$.

Additionally, connected to the $B$-representation, in Section \ref{section:OrbitsA} we will consider a related
representation $(G_A,V_A)$, for which will have $V_A \git G_A \simeq B$
and a commutative diagram
\begin{equation*}
\begin{tikzcd}
\Sel_{\psi} J_b \arrow[r] \arrow[d] & \Sel_{\phi} A_b \arrow[d]  \\
G_B(\Q)\backslash V_{B,b}(\Q) \arrow[r] & G_A(\Q)\backslash V_{A,b}(\Q),
\end{tikzcd}
\end{equation*}
where again the rightmost map is injective and every element in its image
has a representative in $\frac{1}{2}V_{A,b}(\Z)$. It turns out that the
representation $(G_A,V_A)$ is the Vinberg representation associated to
the $\Z/2\Z$-grading on $A_{2n-1}$, which has already been studied in
\cite{SW}. Therefore, in Section \ref{section:mainProof} we can use
the counting results of \emph{loc. cit.} to prove Theorem \ref{theo:mainPhi}.
In Section \ref{section:countingC}, we develop the necessary methods in
geometry-of-numbers to prove the upper bound of Theorem \ref{theo:mainPsi}.
We note that $G_C = \GL_{2n}/\mu_2$ is not semisimple, something that
differs from the cases typically considered in the literature and which
introduces further technical complications. 

In Section \ref{section:tamagawa} we use results from Koymans--Smith \cite{KS}
to prove Theorem \ref{theo:mainMoments}. A simplified account of their
main conjecture would be that ``the moments of Selmer groups asymptotically
grow like the moments of the associated Tamagawa ratios''. These moments
of Tamagawa ratios are much easier to compute: in the case of elliptic
curves we have Tate's algorithm (cf. \cite[\S 7.1]{KS}), and in our situation
we just need to compute the two easiest non-trivial examples 
(Lemmas \ref{lemma:Tamagawa0} and \ref{lemma:tamagawa}).

\subsection{Acknowledgements}
This paper is an adaptation of part of the author's PhD thesis, written under the supervision
of Jack Thorne. I would like to thank him for his many helpful comments
and conversations. I would also like to thank Alex Bartel, Peter Koymans, Jef Laga and
Rong Zhou for interesting discussions related to the contents of this
paper.

The article was completed while the author was supported by the Czech 
Science Foundation GAČR, grant 26-20514S.

\section{Vinberg representations for $B_{2n}$ and $C_{2n}$}
\label{subs:VinbergBC}

We introduce some generalities about Vinberg representations, which will
be later specialised to the case of $\Z/2\Z$-gradings for the Dynkin 
diagrams of type $B_{2n}$ and $C_{2n}$. For more general context, the
reader may consult \cite{Vinberg, Panyushev, Gradings}.

Let $H$ be a connected simple reductive group of adjoint type over a
field $K$ of characteristic zero. Let $\theta \colon H \to H$ be an involution.
By taking differentials, it also induces an involution $d\theta$ on the
Lie algebra $\hh$ of $H$. By considering $\pm 1$ eigenspaces, we get a
decomposition
\[
\hh = \hh(0) \oplus \hh(1),
\]
where $\hh(0) = \hh^{d\theta = 1}$ and $\hh(1) = \hh^{d\theta = -1}$.
We observe that $[\hh(i),\hh(j)] \subset \hh(i+j)$. If we take $G = (H^{\theta})^{\circ}$
and $V = \hh(1)$, we get a representation $(G,V)$ by taking the restriction of
the adjoint representation. Of course, this representation depends on
the choice of $\theta$, but there turns out to be a ``canonical'' choice:
the one that makes the representation \emph{stable} in the following sense:
\begin{defi}
Suppose that $K$ is algebraically closed. Then a vector $v \in V$ is
\emph{stable} if the $G$-orbit of $v$ is closed and its stabiliser
$\Stab_G(v)$ is finite. We say $(G,V)$ is \emph{stable} if $V$ contains
stable vectors. Finally, if $K$ is not necessarily closed, we say that
$(G,V)$ is \emph{stable} if $(G_{K^s},V_{K^s})$ is.
\end{defi}
Stable vectors have a good invariant theory (cf. \cite[Proposition 2.11]{LagaThesis}):
\begin{prop}
Suppose that $\theta \colon H \to H$ is a stable involution, with associated
Vinberg representation $(G,V)$. Then the following properties are satisfied:
\begin{enumerate}
\item Let $\fc \subset V$ be a Cartan subalgebra of $\hh$.
Then the map $N_G(\fc) \to W_{\fc} := N_H(\fc)/Z_H(\fc)$
is surjective. Consequently, the inclusions $\fc \subset V \subset \hh$
induce isomorphisms
\[
\fc \git W_{\fc} \simeq V \git G \simeq \hh \git H.
\]
In particular, the above quotients are isomorphic to an affine space.
\item Let $K$ be algebraically closed and let $x,y \in V(K)$ be regular
semisimple elements. 
Then $x$ and $y$ are $G(K)$-conjugate if and only if they have the same
image in $V \git G$.
\end{enumerate}
\end{prop}

\begin{prop}{\cite[Lemma 2.6]{ThorneThesis}}
Let $K$ be algebraically closed. Then there exists a unique $H(K)$-conjugacy
class of stable involutions $\theta$.
\end{prop}
Even if $K$ is not necessarily algebraically closed, there always exists 
a choice of stable involution $\theta$ defined over $K$. This follows
from \cite[Example 2.20]{LR} and the fact that the outer automorphism
group of both $B_{2n}$ and $C_{2n}$ is trivial. Moreover, $\theta$ can
be constructed explicitly: to do so, fix a pinning $(T,P,\{X_{\alpha}\})$
of $H$, containing:
\begin{itemize}
\item $T \subset H$, a split maximal torus (determining a root system $\Phi_{H}$);
\item $P \subset H$, a Borel subgroup containing $T$ (determining a root basis $S_{H} \subset \Phi_{H}$);
\item $X_{\alpha}$, a generator for $\hh_{\alpha}$ for each $\alpha \in S_{H}$.
\end{itemize}
Let $\check{\rho}$ be the sum of fundamental coweights with respect to $S_{H}$,
and define
\[
\theta = \Ad(\check{\rho}(-1)).
\]
Then it follows from \cite[Corollary 14]{Gradings} that $\theta$ is a
stable grading. We now describe both representations explicitly, which we will
denote as $(G_B,V_B)$ and $(G_C,V_C)$.

\subsection{The $B_{2n}$ representation}
Let $J_m$ denote the $m\times m$ matrix with $1$s in the antidiagonal and
$0$s elsewhere. Define $H = \So(J_m) = \{ A \in \SL_m \mid {}^tAJA = J\}$.
From now on, we will simply denote $\So_m := \So(J_m)$. If $X$ is an $m \times n$
matrix, let us define $X^* = J_n {}^tX J_m$ (if $m = n$, this is simply
reflection across the antidiagonal). We have
\[
\hh = \left\{\begin{pmatrix} B & A \\ -A^* & C \end{pmatrix} \middle| A \in \Mat_{(2n+1)\times 2n}, B\in \Mat_{(2n+1) \times (2n+1)}, C\in \Mat_{2n \times 2n}, B = -B^*, C = -C^* \right\}.
\]
The $2$-grading is given by
\[
\hh(0) = \left\{\begin{pmatrix} B & 0 \\ 0 & C \end{pmatrix} \in \hh \right\}, \quad
\hh(1) = \left\{\begin{pmatrix} 0 & A \\ -A^* & 0 \end{pmatrix} \in \hh \right\}.
\]
The group $G_B$ can be identified with $\So_{2n+1} \times \So_{2n}$, and
$V_B= \hh(1)$ can be identified with $(2n+1) \boxtimes 2n$, the vector space
of $(2n+1) \times 2n$ matrices, and the action of $(g,h) \in G_B$ on a matrix $A \in V_B$
can be given by $gAh^{-1}$.

We note that if the $(2n) \times (2n)$ matrix $A^*A$ has characteristic polynomial
$f(x) = x^{2n}+p_2x^{2n-1}+\dots+p_{4n}$, then the $(2n+1) \times (2n+1)$ 
matrix $AA^*$ has characteristic polynomial $xf(x)$. The coefficients
$p_2,\dots,p_{4n}$ are all invariants of the representation, satisfying $p_{2i}(\lambda v) = \lambda^{2i}p_{2i}(v)$
for all $i = 1,\dots,2n$ and all $\lambda \in K^{\times}$. Let $B := V_B \git G_B = \Spec K[V_B]^{G_B}$ be the GIT
quotient. The following lemma holds due to general facts of Vinberg theory
(see \cite[Corollary 3.6]{Panyushev}):
\begin{prop}
We have that $B \cong \Spec K[p_2,\dots,p_{4n}]$.
\end{prop}
We will write $\pi \colon V_B \to B$ for the invariant map, and we will
write $V_{B,b}(K)$ for those elements in $V(K)$ with invariants $b \in B(K)$.
We will also define the \emph{discriminant} $\Delta$ of an element 
$b \in B(K)$ corresponding to the polynomial $f(x) = x^{2n}+p_2x^{2n-1}+\dots+p_{4n}$ 
as the discriminant of the polynomial $xf(x^2)$, and similarly
define the discriminant of an element $v \in V_B(K)$ as the discriminant
of $\pi(v)$.
We will use the subscript $V_B^{rs}$ to distinguish those elements that 
are regular semisimple (i.e. $\Delta(v) \neq 0$), and we will also write $B^{rs}$ for $\pi(V_B^{rs})$,
which is equivalently the set of elements of $B$ with non-zero discriminant.

\subsection{The $C_{2n}$ representation}
Let $H := \PSp_{4n} = \Sp_{4n}/\mu_2$, where
\[
\Sp_{4n} := \{M \in \Mat_{4n \times 4n} \mid M^t \Omega M = \Omega\}, \quad \Omega = \begin{pmatrix} 0 & J_{2n} \\ -J_{2n} & 0 \end{pmatrix}.
\]
Let $\hh$ be the Lie algebra of $H$; explicitly:
\[
\hh = \left\{X = \begin{pmatrix} B & A \\ C & -B^* \end{pmatrix} \middle| A,B,C \in \Mat_{2n\times 2n}, A = A^*, C = C^* \right\}.
\]
Then the $\Z/2\Z$ grading is:
\[
\hh(0) =  \left\{\begin{pmatrix} B & 0 \\ 0 & -B^* \end{pmatrix} \in \hh \right\}, \quad \hh(1) = \left\{\begin{pmatrix} 0 & A \\ C & 0 \end{pmatrix} \in \hh \right\}.
\]
The Vinberg representation $(G_C,V_C)$ is obtained by taking $G_C = (H^{\theta})^0$,
here identified with $G_C = \GL_{2n}/\mu_2$ and $V_C = \hh(1)$, where $G$
acts by restriction of the adjoint representation.
Explicitly, given $g \in G_C$ and $(A,C) \in V_C$, the action is given by
$g \cdot (A,C) = (gAg^*,(g^{-1})^*Cg^{-1})$.

If the characteristic polynomial of the product of matrices $AC$ is of
the form $x^{2n}+p_{2}x^{2n-1}+\dots+p_{4n}$, then the coefficients
$p_2,\dots,p_{4n}$ are invariant under the action of $G_C$. Let $B := V_C \git G_C$.
Then, by \cite[Corollary 3.6]{Panyushev}, these are all the invariants of the representation,
i.e. $B = \Spec K[p_{2},\dots,p_{4n}]$, and the corresponding ring of
polynomials is freely generated. Note that $V_B \git G_B$ and $V_C \git G_C$
are isomorphic: we denote both spaces by $B$.

Similarly to before, we write $\pi \colon V_C \to B$ for the invariant map, and
$V_{C,b}(K)$ for the elements in $V_C(K)$ mapping to $b \in B(K)$.
The discriminant $\Delta$ of an element $b \in B(K)$ is the same as in
the $B_{2n}$ case, and we similarly define the discriminant of an element
$v \in V_C(K)$ as the discriminant of $\pi(v)$.
We will use the subscript $V_C^{rs}$ to distinguish those elements that 
are regular semisimple.

\section{Orbit parametrisations for $B_{2n}$}
\label{section:OrbitsB}

In this section, we classify the rational and integral orbits of the
representation $(G_B,V_B)$, and we connect these orbits to elements of
the Selmer group $\Sel_{\psi}J$ defined in the introduction. To lighten the notation, for this section
let us denote $(G,V) = (G_B,V_B)$. We fix throughout a field $K$ of characteristic
$0$.

\subsection{Rational orbits}

In general, a $G(K^s)$-orbit of elements in $V(K)$ might break up into
multiple $G(K)$-orbits in $V(K)$. We have the following general result from
arithmetic invariant theory (see \cite[Proposition 1]{AIT}) 
indicating how this phenomenon can be studied with Galois cohomology groups:
\begin{prop}
\label{prop:ait}
Let $v \in V(K)$. The set of $G(K)$-orbits in $V(K)$ which are $G(K^s)$-conjugate
to $v$ is in bijection with the kernel of the map
\[
H^1(K,\Stab_G(v)) \to H^1(K,G)
\]
of pointed sets.
\end{prop}

In this section, we will first explain how to construct a given ``distinguished''
orbit $v$, and then we will show how to construct from it all the rational orbits with given
invariants from an element $\alpha \in \ker(H^1(K,\Stab_G(v)) \to H^1(K,G))$.

\subsubsection*{A distinguished orbit}
\label{subsubs:dist}

Let $b = (p_2,\dots,p_{4n}) \in B(K)$, and consider the polynomial
$f(x) = x^{2n}+p_2x^{2n-1}+\dots+p_{4n}$. Define the $K$-vector space $M = K[x]/(xf(x^2)) = K[\beta]$,
where $\beta$ denotes the image of $x$ inside $M$. Note that $M$ is spanned by the $K$-linear combinations of $1,\beta,\dots,\beta^{4n}$.
Define the bilinear form $(\cdot,\cdot) \colon M \times M \to K$ by:
\[
(\lambda,\mu) = \text{ coefficient of }\beta^{4n} \text{ in }\lambda\mu.
\]
Let $L_1 = K[x]/(xf(x))$ and let $L_2 = K[x]/(f(x))$. Then, $M$ is isomorphic
to $L_1 \oplus \beta L_2$, where there is a natural inclusion $L_1 \xhookrightarrow{} M$ 
by sending $x \mapsto x^2$. In other words, $L_1$ is the subspace spanned by
$\{1,\beta^2,\dots,\beta^{4n}\}$ and $\beta L_2$ is spanned by $\{\beta,\beta^3,\dots,\beta^{4n-1}\}$.
Then, the form $(\cdot,\cdot)$ splits as a direct sum of bilinear
forms in $L_1$ and $\beta L_2$. Using the explicit power bases, we can see
that both quadratic forms on $L_1$ and $L_2$ have discriminant $1$
and are in fact split, so we can isometrically identify $L_1$ with a
quadratic space $(W_1,J_{2n+1})$ of dimension $2n+1$ and $L_2$ with a
quadratic space $(W_2,J_{2n})$ of dimension $2n$.

Let $W$ be the quadratic space given by $(W_1,J_{2n+1}) \oplus (W_2,J_{2n})$,
and consider the multiplication-by-$\beta$ map $T_{\beta} \colon M \to M$, 
which can also be seen as a map from $W$ to $W$. Given that $T_{\beta}$
is self-adjoint with respect to $(\cdot,\cdot)$, we get that the matrix of
$T_{\beta}$ on $W$ is of the form
\[
  \left(\begin{array}{ c | c }
    0_{(2n+1)\times(2n+1)} & A \\
    \hline
    A^* & 0_{2n\times 2n}
  \end{array}\right),
\]
where $A \in \Mat_{(2n+1)\times (2n)}$. Thus, we get an element $v \in V(K)$
with invariants $b \in B(K)$ by construction. We also observe the following:
\begin{prop}
Let $v \in V(K)$ be the orbit previously constructed, and assume that
$\Delta(v) \neq 0$. Then, the stabiliser $\Stab_G(v)$ is isomorphic to
the kernel of the norm map $\Res_{L_2/K} \mu_2 \to \mu_2$.
\end{prop}
\begin{proof}
Given that $v$ is regular semisimple, the centraliser of $T_{\beta}$
in $\GL(M)$ is $M^{\times}$. Since the centraliser actually lies inside
$\So(M)$, this forces elements $\lambda \in M^{\times}$ to satisfy $\lambda^2 = 1$.
Moreover, because $\lambda$ needs to preserve both $L_1$ and $\beta L_2$,
we see that $\lambda \in L_1^{\times}$. Finally, the fact that $\lambda \in \So(W_1) \times \So(W_2)$
forces $N_{L_1/K}(\lambda) = 1$ and $N_{L_2/K}(\ol{\lambda}) = 1$, where
$\ol{\lambda}$ is the image of $\lambda$ in $L_2$.

The conclusion is that the stabiliser is in bijection with the set of elements
of $ \Res_{L_1/K} \mu_2$ whose norm is $1$ and whose image in $L_2$ also
has norm $1$. This can be identified with $\ker(\Res_{L_2/K} \mu_2 \to \mu_2)$,
and so we are done.
\end{proof}

We will denote the kernel of the map $\Res_{L_2/K} \mu_2 \to \mu_2$ by 
$(\Res_{L_2/K} \mu_2)_{N = 1}$.

\subsubsection*{The other orbits}
\label{subsubs:other}

Let $G' = \So_{2n+1}\times \mathrm{O}_{2n}$. We will start by explaining
how to construct all the different $G'(K)$-orbits, and then we will specialise
to $G(K)$-orbits. Note that given $v \in V(K)$ with $\Delta(v) \neq 0$,
we have that $\Stab_{G'}(v) \cong \Res_{L_2/K}\mu_2$, and that
$H^1(K,\Res_{L_2/K}\mu_2) \cong L_2^{\times}/L_2^{\times 2}$. We also observe
that the pointed set $H^1(K,\So_{m})$ parametrises non-degenerate quadratic spaces of dimension $m$ and
discriminant $1$, and that the trivial element of $H^1(K,\So_{m})$ corresponds
to the (unique) split orthogonal space of dimension $m$. Similarly,
the pointed set $H^1(K,\mathrm{O}_{m})$ classifies non-degenerate 
quadratic spaces of dimension $m$, with a similar trivial element.
The map $H^1(K,\So_m) \to H^1(K,\mathrm{O}_m)$ has trivial kernel as a
map of pointed sets, as can be seen from the usual long exact sequence
in group cohomology of
\[
\begin{tikzcd}
1 \arrow[r] & \So_m \arrow[r] & \mathrm{O}_m \arrow[r, "det"] & \{\pm 1\} \arrow[r] & 1.
\end{tikzcd}
\]
In fact, $H^1(K,\So_m) \to H^1(K,\mathrm{O}_m)$ is injective (cf. \cite[\S 29.E]{Involutions}).

Given $\alpha \in (L_2^{\times}/L_2^{\times 2})$ mapping to the trivial 
element in $H^1(K,G')$, we will show how to construct a rational orbit from it.
An element $\alpha \in L_2^{\times}$ can be lifted to an element of 
$L_1^{\times} \cong K^{\times} \times L_2^{\times}$ by simply considering
$(1,\alpha) \in L_1^{\times}$. Moreover, as in last section, we can naturally
embed $L_1 \xhookrightarrow{} M$, so given $\alpha \in (L_2^{\times}/L_2^{\times 2})$
we can naturally consider it as an element of $M$. Under this identification,
consider the quadratic form $(\cdot,\cdot)_{\alpha} \colon M \times M \to K$
defined by
\begin{equation}
\label{eq:form}
(\lambda,\mu)_{\alpha} = \text{ coefficient of }\beta^{4n} \text{ in }\alpha^{-1}\lambda\mu.
\end{equation}
As before, this quadratic form splits as a direct sum of quadratic forms
in $L_1$ and $\beta L_2$. If $\alpha$ has norm $1$ (up to squares) in $L_2$,
then both forms have discriminant $1$, so they give
a well-defined map to $H^1(K,G')$. Unwinding the definitions similarly to \cite[\S 5]{AIT},
the condition that $\alpha$ lands in the kernel of $H^1(K,G')$ translates
precisely to both forms $(\cdot,\cdot)_{\alpha}|_{L_1}$ and $(\cdot,\cdot)_{\alpha}|_{\beta L_2}$
being split of discriminant $1$. Therefore, under appropriate change of bases in
$L_1$ and $L_2$, the map $T_{\beta}$ given an element of $V(K)$ in the
same way as in the distinguished case.

Thus, given $\alpha \in (L_2^{\times}/L_2^{\times 2})_{N \equiv 1}$ which
maps to the trivial element of $H^1(K,G')$, we have constructed a rational
$G'(K)$-orbit.
We now turn our attention to $G(K)$-orbits. Following \cite[\S 4.3]{AIT2},
there is a map 
\[
H^1(K,(\Res_{L_2/K} \mu_2)_{N = 1}) \xrightarrow{} (L_2^{\times}/L_2^{\times 2})_{N \equiv 1}
\]
which is either bijective or
$2$-to-$1$, according to whether $f(x)$ has an odd degree factor over $K$
or not, which in turn depends on whether $L_2^{\times}[2]$ has an element of
norm $-1$ or not. Therefore, a $G'(K)$-orbit in $V(K)$ splits in either
one or two $G(K)$-orbits. 
In the case where $f(x)$ does not have an odd degree factor over $K$, we
note that the stabiliser over $K$ of the constructed $v$ in $\So_{2n+1} \times \So_{2n}$
is the same as the stabiliser in $\So_{2n+1} \times \mathrm{O}_{2n}$.
By choosing $h \in \mathrm{O}_{2n}(K) \setminus \So_{2n}(K)$, we can obtain
a new orbit by just considering the element $(1,h) \cdot v$. If $f(x)$
has an odd degree factor over $K$, these two constructed orbits coincide.
We summarise our results as follows:
\begin{theorem}
\label{theo:Qorbits}
Let $b \in B(K)$ with $\Delta(b) \neq 0$. Then the set of $G(K)$-orbits
in $V_b(K)$ is in bijection with the set of equivalence classes $(\alpha,s)$,
where $\alpha \in (L_2^{\times})_{N\equiv1}$ maps to the
trivial element in $H^1(K,G)$ and $s \in K^{\times}$ satisfies $N(\alpha) = s^2$. Two pairs $(\alpha,s)$ and
$(\alpha',s')$ are equivalent if there exists $c \in L_2^{\times}$ such
that $\alpha' = c^2\alpha$ and $s' = N(c) s$. The stabiliser of such an
orbit is isomorphic to $(\Res_{L_2/K}\mu_2)_{N = 1}$.
\end{theorem}

\subsection{Connection with hyperelliptic curves}

Let $b \in B^{rs}(K)$ correspond to the polynomial $f(x) \in K[x]$ with 
$\deg f = 2n$. Consider the hyperelliptic curve
$C_b \colon y^2 = xf(x)$ and its Jacobian $J_b := \Jac(C_b)$. 
If the discriminant of $xf(x)$ is non-zero, then the roots $x_0 = 0, x_1,\dots,x_{2n}$
of $xf(x)$ over $\ol{K}$ are all different. If we denote $P_i = (x_i,0)$
and $\infty$ is the point at infinity, then $J_b[2](\ol{K})$ is generated 
by the elements $[(P_i) - \infty]$, with the only relation that $\sum_{i=0}^{2n}[(P_i)-\infty] = 0$.

Consider the order $2$ subgroup $M \subset J_b[2]$ generated by $P_0 = (0,0)$,
and consider its orthogonal complement $M^{\perp}$ under the Weil pairing.
For convenience, we can give an explicit description of $M^{\perp}(\ol{K})$:
every element of $J_{b}[2]$ can be written uniquely as a sum $[(P_0)-\infty] + \sum_{i \in I} [(P_i)-\infty]$
for some subset $I \subset \{1,\dots,2n\}$. Then, $M^{\perp}(\ol{K})$ consists
of elements of this form such that $|I|$ is even. Note that $M^{\perp}(\ol{K})$
has size $2^{2n-1}$ and that $M \leqslant M^{\perp}$.
We can construct some isogenies as in \eqref{eq:JA}:
\[
J_b \xrightarrow{\phi_M} A_b \xrightarrow{\phi} A_b^{\vee} \xrightarrow{\phi_M^{\vee}} J_b,
\]
with $J_b[\phi_M] = M$, if $\psi = \phi \circ \phi_M$ then $J_b[\psi] = M^{\perp}$
and the whole composition is the multiplication-by-$2$ map. 

\begin{prop}
\label{prop:StabJ}
Let $v \in V(K)$ with $\Delta(v) \neq 0$. Then $\Stab_G(v) \cong J_b[\psi]$.
\end{prop}
\begin{proof}
It suffices to show that $J_b[\psi] \cong (\Res_{L_2/K}\mu_2)_{N = 1}$,
which is an elementary computation.
\end{proof}

Note that we also have that we have the injective descent map $A_b^{\vee}(K)/\psi(J_b(K)) \xhookrightarrow{} H^1(K,J_b[\psi])$.
It is then natural to ask whether the elements of $A_b^{\vee}(K)/\psi(J_b(K))$ 
actually correspond to $G_b(K)$-orbits in $V_b(K)$. We now state the main
theorem of this section.
\begin{theorem}
\label{theo:comp}
The natural composition
\[
A_b^{\vee}(K)/\psi(J_b(K)) \xrightarrow{\eta_b} H^1(K,J_b[\psi]) \to H^1(K,G)
\]
is trivial.
\end{theorem}
\begin{defi}
\label{defi:soluble}
We will say an element $v \in V^{rs}(K)$ is \emph{$K$-soluble} if $v \in \eta_b(A_b^{\vee}(K)/\psi(J_b(K)))$.
\end{defi}

It is not so obvious what an explicit description of the map $A_b^{\vee}(K)/\psi(J_b(K)) \to H^1(K,J_b[\psi])$
should be. However, we can try to simplify the situation by trying
to relate it to the $2$-descent map $J_b(K)/2J_b(K) \to H^1(K,J_b[2])$.
Consider the group $G' = \So_{2n+1} \times \mathrm{O}_{2n}$: similarly
to Proposition \ref{prop:StabJ}, we can see that $\Stab_{G'}(v) \cong \Res_{L_1/K}(\mu_2) \cong J_b[2]$.
We then have the following commutative diagram:
\begin{equation}
\label{eq:diagram}
\begin{tikzcd}
A_b^{\vee}(K)/\psi(J_b(K)) \arrow[r] \arrow[d, "\delta_\psi"] & J_b(K)/2J_b(K) \arrow[d, "\delta_2"]  \\
H^1(K,J_b[\psi]) \arrow[r, "\iota"] \arrow[d] & H^1(K,J_b[2])  \arrow[d] \\
H^1(K,G) \arrow[r] & H^1(K,G')
\end{tikzcd}
\end{equation}
The map $H^1(K,J_b[2]) \to H^1(K,G') \cong H^1(K,\So_{2n+1}) \times H^1(K,\mathrm{O}_{2n})$
can be given using the same recipe as in Section \ref{subsubs:other}. 
Explicitly, given $\alpha \in H^1(K,J_b[2])$,
which can be viewed both as an element of $L_2^{\times}/L_2^{\times 2}$ and
as an element of $(L_1^{\times}/L_1^{\times 2})_{N \equiv 1}$ via 
$\alpha \mapsto (N_{L_2/K}(\alpha),\alpha) \in K^{\times} \times L_2^{\times} \cong L_1^{\times}$,
we obtain two quadratic spaces $(\cdot,\cdot)_{\alpha}^{(1)} \colon L_1 \times L_1 \to K$ and $(\cdot,\cdot)_{\alpha}^{(2)}\colon L_2 \times L_2 \to K$
given by
\[
(\mu,\lambda)_{\alpha}^{(1)} = \text{ coefficient of }\beta_1^{2n}\text{ in } \alpha^{-1}\mu\lambda \text{ (inside }L_1\text{)}
\]
and
\[
(\mu,\lambda)_{\alpha}^{(2)} = \text{ coefficient of }\beta_2^{2n-1}\text{ in } \alpha^{-1}\mu\lambda \text{ (inside }L_2\text{)},
\]
where we are writing $L_1 = K\langle 1,\beta,\dots,\beta^{2n} \rangle$
and $L_2 = K\langle 1,\beta,\dots,\beta^{2n-1} \rangle$.
Alternatively, if we consider the codimension $1$ vector subspace $\beta_1 L_1$
of $L_1$, we have that $(\cdot,\cdot)_{\alpha}^{(2)}$ is equivalent
to a form $(\cdot,\cdot)_{\alpha}^{(2')}\colon \beta_1 L_1 \times \beta_1 L_1 \to K$ given
by $(\beta_1\mu,\beta_1\lambda)_{\alpha}^{(2')} := (\mu,\beta_1\lambda)_{\alpha}^{(1)}$
(we can check that this is well-defined).

The image of $\alpha \in H^1(K,J_b[2])$ in $H^1(K,\So_{2n+1}) \times H^1(K,\mathrm{O}_{2n})$
is given by the quadratic spaces $(L_1,(\cdot,\cdot)_{\alpha}^{(1)})$ and
$(L_2,(\cdot,\cdot)_{\alpha}^{(2)})$, and these quadratic spaces will
correspond to the trivial element if and only if they are split of discriminant $1$.
We note that the discriminant of $(\cdot,\cdot)_{\alpha}^{(1)}$ is $1$, 
while the discriminant of $(\cdot,\cdot)_{\alpha}^{(2)}$ is equal to $N_{L_2/K}(\alpha)$.
Therefore, it is not necessarily the case that the composition $J_b(K)/2J_b(K) \xrightarrow{\delta_2} H^1(K,J_b[2]) \to H^1(K,G')$
is trivial: it is a necessary condition that $N_{L_2/K}(\alpha) \in K^{\times 2}$.

Recall that there is a surjective map $H^1(K,J_b[\psi]) \to (L_2^{\times}/L_2^{\times 2})_{N \equiv 1}$,
which is either bijective or $2$-to-$1$. Then, the map
$\iota \colon H^1(K,J_b[\psi]) \to H^1(K,J_b[2]) \cong L_2^{\times}/L_2^{\times 2}$
is just given by the natural inclusion $H^1(K,J_b[\psi]) \to (L_2^{\times}/L_2^{\times 2})_{N \equiv 1} \to L_2^{\times}/L_2^{\times 2}$.
\begin{lemma}
\label{lemma:comp}
Let $[D] \in J_b(K)/2J_b(K)$, and suppose that $\delta_2([D]) \in \im(\iota)$.
Then the image of $\delta_2([D])$ in $H^1(K,G')$ is trivial.
\end{lemma}
\begin{proof}
We start by recounting the proof of \cite[Proposition 5.2]{BG}. Consider the two quadrics in
$L_1 \oplus K$ given by
\[
Q_1(\lambda,a) = (\lambda,\lambda)_{\alpha}^{(1)}, \quad Q_2(\lambda,a) = (\lambda,\beta_1\lambda)_{\alpha}^{(1)}+a^2.
\]
Then, it is shown in loc. cit. that there exists a rational $n$-dimensional subspace
$Y$ of $L_1 \oplus K$ which is isotropic with respect to both $Q_1$ and $Q_2$.
In particular, given that the line $0 \oplus K$ is not contained in $Y$,
we see that the projection of $Y$ to $L_1$ is an $n$-dimensional isotropic subspace of
$Q_1$, thus showing that $(\cdot,\cdot)_{\alpha}^{(1)}$ is split.

Now, consider the subspace $Y' = Y \cap (L_1 \oplus 0)$, of dimension at least $n-1$.
We see that $Y' \cap f(\beta_1)L_1 = \{0\}$, as $(f(\beta_1),f(\beta_1))_{\alpha}^{(1)} = N_{L_2/K}(\alpha^{-1})N_{L_2/K}(\beta_1) \neq 0$.
Therefore, the subspace $\beta_1 Y'$ of $\beta _1L_1$ has dimension at least
$n-1$, and it is also the case that for any $\beta_1\mu,\beta_1\lambda \in \beta_1 Y'$
we have that $(\beta_1\mu,\beta_1\lambda)_{\alpha}^{(2')} = (\mu,\beta_1\lambda)_{\alpha}^{(1)} = 0$ by construction.
Therefore, $(\cdot,\cdot)_{\alpha}^{(2)}$ has a rational isotropic space
of dimension $n-1$ and thus, as a quadratic space, we have that $(L_2,(\cdot,\cdot)_{\alpha}^{(2)}) \cong H^{n-1} \oplus V'$,
where $H \sim \left< \begin{pmatrix} 0 & 1 \\ 1 & 0 \end{pmatrix} \right>$.
But $V'$ is a quadratic space of dimension $2$ and discriminant $1$ (by hypothesis), and
therefore $V' \sim H$ as well, showing that $(\cdot,\cdot)_{\alpha}^{(2)}$ is split,
as wanted.
\end{proof}

\begin{proof}[Proof of Theorem \ref{theo:comp}]
First, note that the map $H^1(K,G) \to H^1(K,G')$ has a trivial pointed kernel, which is
equivalent to $H^1(K,\So_{2n}) \to H^1(K,\mathrm{O}_{2n})$ having a trivial pointed kernel,
as noted in Section \ref{subsubs:other}. Then, the proof follows from 
Lemma \ref{lemma:comp} and the commutativity of the diagram \eqref{eq:diagram}.
\end{proof}

\begin{obs}
Let $A_b^{\vee}[\hat{\psi}] = \{0,T_A\}$. Then, both $0$ and $T_A$ give distinguished
orbits of $G(K) \backslash V_b(K)$. Whether or not these two orbits coincide
depends on whether $T_A \in \psi(J_b(K))$.
\end{obs}

\begin{cor}
\label{cor:Selmer}
Let $K$ be a number field, and let $b \in B(K)$ with $\Delta(b) \neq 0$. Then there is an embedding
\[
\Sel_{\psi}(J_b) \xhookrightarrow{} G(K) \backslash V_b(K).
\]
\end{cor}
\begin{proof}
Consider the commutative diagram
\[
\begin{tikzcd}
A_b^{\vee}(K)/\psi(J_b(K)) \arrow[r] \arrow[d] & H^1(K,J_b[\psi]) \arrow[r] \arrow[d] & H^1(K,G)\arrow[d] \\
\prod_v A_b^{\vee}(K_v)/\psi(J_b(K_v)) \arrow[r, "\delta_{\psi,v}"] & \prod_v H^1(K_v,J_b[\psi]) \arrow[r] & \prod_v H^1(K_v,G),
\end{tikzcd}
\]
where the product is taken over all finite places $v$ of $K$. Recall
that $\Sel_\psi(J_b)$ is defined as the kernel of the map $H^1(K,J_b[\psi]) \to \prod_v H^1(K_v,J_b[\psi])/(\im (\delta_{\psi,v}))$.
Our statement then follows from the fact that the composition of maps in
the second row is trivial by Theorem \ref{theo:comp}, and the fact that
the map $H^1(K,G) \to \prod_v H^1(K_v,G)$ has trivial kernel by \cite[Proposition 6.8]{LagaThesis}.
\end{proof}

\subsection{Integral orbits}

To prove our main theorems, we will require an integral version of 
Corollary \ref{cor:Selmer}. We remark that even though we have originally
constructed our representation over $K$, a field of characteristic zero,
we could also have constructed $(G,V)$ over $\Z$. In this case, we
also have $V \git G = B = \Spec \Z[p_2,\dots,p_{4n}]$.

\begin{theorem}
\label{theo:integral}
Every element in the image of the map
\[
\Sel_{\psi}(J_b) \xhookrightarrow{} G(\Q) \backslash V_b(\Q).
\]
has a representative in $\frac{1}{2}V_b(\Z)$.
\end{theorem}

Because $G$ has class number $1$ (cf. \cite[Proposition 7.2]{LagaThesis}),
it will suffice to see that the map
\[
A_b^{\vee}(\Q_p)/\psi(J(\Q_p)) \xhookrightarrow{} G(\Q_p) \backslash V_b(\Q_p)
\]
falls inside the image of the inclusion map $\frac{1}{2}V_b(\Z_p) \to G(\Q_p) \backslash V_b(\Q_p)$ 
for all primes $p$. We start by giving an ideal parametrisation
of integral orbits inside $\Z_p$, in an analogous way to other results in
the literature, such as \cite[Proposition 6.7]{ShankarD2n+1}.
For the proof of Theorem \ref{theo:integral}, we will need to know when 
a $\Z_p$-lattice $L$ of dimension $m$ with a given symmetric bilinear
form $L \times L \to \Z_p$ is isometric over $\Z_p$ to an $m$-dimensional 
lattice $L_m$ with matrix $J_m$. We summarise known results on this
next lemma:
\begin{lemma}
\label{lemma:latticep2}
Let $I$ be a free $\Z_p$-module of rank $m$ equipped with a symmetric bilinear
form $\varphi \colon I \times I \to \Z_p$ of discriminant $1$.
\begin{itemize}
\item If $p \neq 2$, then $I$ is isometric to $L_m$ over $\Z_p$.
\item If $p = 2$ and $m = 2m'+1$, then $I$ is isometric to $L_m$ over $\Z_2$
if $I \otimes \Q_2$ is a split orthogonal space.
\item If $p = 2$ and $m = 2m'$, then $I$ is isometric to $L_m$ over $\Z_2$
if $I \otimes \Q_2$ is a split orthogonal space and $I$ is an even lattice
(i.e. $\varphi(x,x) \in 2\Z_2$ for all $x \in I$).
\end{itemize}
\end{lemma}
In the last item, the condition that $I$ is an even lattice is necessary,
as $(L_m,J_m)$ admits both an even and an odd lattice over $\Z_2$. These
two lattices can be transformed to one another via an element of
$\mathrm{O}_m(\Q_2)$ with coefficients in $\frac{1}{2}\Z_2$.

\begin{prop}
\label{prop:intOrbits}
Let $b \in B(\Z_p)$ with $\Delta(b) \neq 0$. Then the set of orbits
$G(\Z_p)\backslash V_b(\Z_p)$ is in bijection with the set of equivalence classes
of $(I_1,I_2,\alpha,s)$, where $I_1$ is a fractional ideal of $R_1$, $I_2$ is 
a fractional ideal of $R_2$, $\alpha \in (L_2^{\times})_{N\equiv 1}$ and
$s \in \Q_p^\times$; satisfying:
\begin{enumerate}
\item $I_1^2 \subset \alpha R_1$ and $N(I_1)^2 = N_{L_1/\Q_p}(\alpha)$, where
$\alpha$ can be interpreted as an element of $L_1 \cong \Q_p \times L_2$ via
$\alpha \mapsto (N_{L_2/\Q_p}(\alpha),\alpha)$.
\item $I_2^2 \subset \alpha R_2$ and $N(I_2)^2 = N_{L_2/\Q_p}(\alpha)$.
\item Let $\ol{I_1}$ denote the projection of $I_1$ in $L_2$, and let
$\ol{I_1}' = \{\gamma \in L_2 \mid (0,\beta_1\gamma) \in I_1\}$.
Then $\ol{I_1} \subset I_2 \subset \ol{I_1}'$.
\item The forms $(\cdot,\cdot)_{\alpha}^{(1)}$ and $(\cdot,\cdot)_{\alpha}^{(2)}$
are split of discriminant $1$ over $\Q_p$.
\item $I_2$ is even with respect to $(\cdot,\cdot)_{\alpha}^{(2)}$.
\item $N_{L_2/\Q_p}(\alpha) = s^2$.
\end{enumerate}
Two such tuples $(I_1,I_2,\alpha,s)$ and $(I_1',I_2',\alpha',s')$ are equivalent
if and only if there exists an element $c \in L_2^{\times}$ such that
$I_1 = cI_1'$, $I_2 = cI_2'$, $\alpha = c^2\alpha'$ and $s = N_{L_2/\Q_p}(c)s'$.
An integral orbit $(I_1,I_2,\alpha,s)$ corresponds to the rational orbit
given by $(\alpha,s)$. 
\end{prop}
\begin{proof}
First, we start with a tuple $(I_1,I_2,\alpha,s)$ and we construct an
orbit in $G(\Z_p)\backslash V_b(\Z_p)$. First, we note that the forms
$(\cdot,\cdot)_{\alpha}^{(1)}$ and $(\cdot,\cdot)_{\alpha}^{(2)}$, when
restricted to $I_1$ and $I_2$ respectively, take integral values and
are split of discriminant $1$. Additionally, $I_2$ is even with respect
to $(\cdot,\cdot)_{\alpha}^{(2)}$. Therefore, by Lemma \ref{lemma:latticep2} we can find $\Z_p$-bases for 
$I_1$ and $I_2$ such that the forms have Gram matrices $J_{2n+1}$ and $J_{2n}$
respectively. Then, also by construction we have that the matrix of $T_{\beta}$ 
has values in $\Z_p$, so it gives an element of $V_b(\Z_p)$.

Now, suppose that we start with an orbit in $G(\Z_p)\backslash V_b(\Z_p)$.
Theorem \ref{theo:Qorbits} gives $(\alpha,s)$ and hence properties 4 and 6. We
recall that such an orbit can be constructed as the matrix of $T_{\beta}$
in $M = \Q_p[x]/(xf(x^2))$. Given a basis $\{e_1,\dots,e_{4n+1}\}$ of $M$,
the action of $T_{\beta}$ realises $J = \Z_p\langle e_1,\dots,e_{4n+1}\rangle$
as an $R = \Z_p[x]/(xf(x^2))$-submodule. Note that $R \cong R_1 \oplus \beta R_2$.
The fact that $T_{\beta}$ respects
$R_1$ and $R_2$ implies that $J = I_1 + \beta I_2$ for some fractional
ideals $I_1$ in $R_1$ and $I_2$ in $R_2$, which necessarily satisfy 
$\ol{I_1} \subset I_2 \subset \ol{I_1}'$. The fact that the
forms $(\cdot,\cdot)_{\alpha}^{(1)}$ and $(\cdot,\cdot)_{\alpha}^{(2)}$
have to be self-dual with respect to $I_1$ and $I_2$ with matrices isometric
to $J_{2n+1}$ and $J_{2n}$ over $\Z_p$, respectively, give the rest of
the hypotheses.

These two constructions are inverse to each other, and so we are done.
\end{proof}

\begin{proof}[Proof of Theorem \ref{theo:integral}]
It suffices to show that for every element of $\hat{A_b}(\Q_p)/\psi(J_b(\Q_p))$
there is a tuple $(I_1,I_2,\alpha,s)$ satisfying the conditions
of Proposition \ref{prop:intOrbits}.
We note that splitness of the forms over $\Q_p$ follows from Theorem
\ref{theo:comp}. Furthermore, by \cite[Lemma 4.9]{LTA2n} (cf. \cite[Proposition 8.5]{BG}),
there exists a fractional ideal $I_1$ in $R_1$ such that $I_1^2 \subset \alpha R_1$
with $N(I_1)^2 = N_{L_1/\Q_p}(\alpha)$. 

We can observe that when taking the image under the tautological map
$L_1 \to L_2$, the lattices $\ol{I_1}$ and 
$\ol{I_1}'$ are dual to each other with respect to the form 
$(\cdot,\cdot)_{\alpha}^{(2)}$. This follows from observing that for any
$\mu,\lambda \in L_2$ with liftings $\mu',\lambda' \in L_1$ we have that
\[
(\mu,\lambda)_{\alpha}^{(2)} = (\mu',\beta\lambda')_{\alpha}^{(1)}.
\]
Then, the process of finding a fractional
ideal $I_2$ with the required conditions reduces to finding a lattice
$\ol{I_1} \subset \Lambda \subset \ol{I_1}'$ which is self-dual
and is stable under multiplication by $\beta_2$, up to considerations at $p = 2$. We further
observe that any lattice $\Lambda$ satisfying $\ol{I_1} \subset \Lambda \subset \ol{I_1}'$
is automatically stable under $\times \beta_2$, so it automatically is
a fractional ideal.

We split our the rest of our proof into the two cases $p \neq 2$ and $p = 2$,
in a similar way to \cite[Propositions 6.9, 6.11]{ShankarD2n+1}.
\begin{itemize}
\item $p \neq 2$: By \cite[Lemma 3.4]{Cassels} we can find a basis $(f_i)$ of $\ol{I_1}$
such that its Gram matrix with respect to $(\cdot,\cdot)_{\alpha}^{(2)}$
is
\[
\begin{pmatrix}
u_1p^{a_1} & & & \\
& u_2p^{a_2} & & \\
& & \ddots & \\
& & & u_{2n}p^{a_{2n}}\\
\end{pmatrix}
\]
where the $a_i$ are non-negative integers and $u_i \in \Z_p^{\times}$.
By replacing $f_i$ with $p^{-\lfloor \frac{a_i}{2} \rfloor}f_i$, we may
assume that $a_i \in \{0,1\}$, and the resulting lattice $\Lambda$ still
satisfies $\ol{I_1} \subset \Lambda \subset \ol{I_1}'$.
Write $\Lambda = \Lambda_0 \oplus \Lambda_1$, where $\Lambda_i$ is the span
of those $f_j$ with $b_j = i$ ($i = 0,1$). Given that the discriminant of
the form is $1$ modulo squares, the dimensions of both $\Lambda_0$ and
$\Lambda_1$ have to be even. We will now see that both $\Lambda_0 \otimes \Q_p$
and $\Lambda_1 \otimes \Q_p$ are split quadratic spaces.

Let $\Lambda_0$ be spanned by $(f_1,\dots,f_{2a})$ and
let $\Lambda_1$ be spanned by $(f_{2a+1},\dots,f_{2n})$.
Then, the discriminant of $\Lambda_0 \otimes \Q_p$ is $(-1)^a \prod_{i=1}^{2a} u_i$
and the Hasse invariant is $1$, as $(u_i,u_j)_p = 1$ for all $u_i,u_j \in \Z_p^{\times}$.
On the other hand, the discriminant of $\Lambda_1\otimes \Q_p$ is $(-1)^{n-a} \prod_{i=2a+1}^{2n}u_i$
and its Hasse invariant is $(-1)^{(n-a)(p-1)/2} \prod_{i=2a+1}^{p} \lp \frac{u_i}{p}\rp$.
A straightforward computation shows that the Hasse invariant of $(\Lambda_0 \oplus \Lambda_1) \otimes \Q_p$
is equal to the Hasse invariant of $\Lambda_1 \otimes \Q_p$, so both
these invariants are equal to $1$. Given that $(-1)^{(n-a)(p-1)/2}$ is 
equal to $(-1)^{n-a}$ up to squares (indeed, both these quantities are
equal to $(-1)^{n-a}$ if $p \equiv 3 \pmod{4}$ or equal to $1$ modulo
squares if $p \equiv 1 \pmod{4}$), this forces the discriminant of $\Lambda_1 \otimes \Q_p$
to be equal to $1$. Since the discriminant of $(\Lambda_0 \oplus \lambda_1) \otimes \Q_p$
is $1$, and also the product of discriminants in $\Lambda_0$ and $\Lambda_1$,
this implies that the discriminant of $\Lambda_0 \otimes \Q_p$
is also $1$, proving our claim that both $\Lambda_0 \otimes \Q_p$
and $\Lambda_1 \otimes \Q_p$ are split quadratic spaces.

Thus, we can choose a basis of $\Lambda_0$ such that its Gram matrix is
$J_{2a}$, and we can choose a basis of $\Lambda_1$ such that its Gram matrix
is $pJ_{2(n-a)}$. By replacing the elements of the basis $f_{2a+1},\dots,f_{2n}$ by $f_{2a+1}/p, \dots, f_{a+n}/p, f_{a+n+1}, \dots, f_{2n}$,
we get the matrix $J_{2(n-a)}$. Therefore, we obtain a self-dual lattice
$\Lambda = \Lambda_0 \oplus \Lambda_1$ with the desired inclusion conditions.

\item $p = 2$: In this situation, by \cite[Lemma 4.1]{Cassels} we can find a basis of $\ol{I_1}$ such
that its Gram matrix with respect to $(\cdot,\cdot)_{\alpha}^{(2)}$ is
\[
\begin{pmatrix}
2^{a_1}Q_1 & & & \\
& 2^{a_2}Q_2 & & \\
& & \ddots & \\
& & & 2^{a_{k}}Q_k\\
\end{pmatrix},
\]
where $a_i \geq 0$ and the $Q_i$ are either $1 \times 1$ matrices with
an entry in $\Z_p^{\times}$ or $2 \times 2$ matrices of the form
\[
Q_i = \begin{pmatrix} 0 & 1 \\ 1 & 0 \end{pmatrix} \quad \text{or} \quad Q_i = \begin{pmatrix} 2 & 1 \\ 1 & 2 \end{pmatrix}.
\]
As before, we may assume that $a_i \in \{0,1\}$. For the $2 \times 2$ matrices,
we may further assume that $a_i = 0$: if $a_i = 1$, we may substitute $e_1$
for $e_1/2$ to get a self-dual lattice. Therefore, we may assume that the
Gram matrix is of the form
\[
\begin{pmatrix}
2U & & & \\
& Q_2 & & \\
& & \ddots & \\
& & & Q_k\\
\end{pmatrix},
\]
where $U$ is a diagonal matrix of size $2a \times 2a$ with unit entries,
and the $Q_i$ are either $1 \times 1$ or $2 \times 2$ matrices with unit
determinant. Finally, we notice that for a matrix $\begin{pmatrix} 2u_1 & 0 \\ 0 & 2u_2 \end{pmatrix}$
with $u_1,u_2 \in \Z_p^{\times}$, the basis spanned by $(e_1+e_2)/2$ and
$(e_1-e_2)/2$ gives a self-dual lattice. We can conclude that there exists
a self-dual lattice $\ol{I_1} \subset \Lambda \subset \beta_2^{-1}\ol{I_1}$.

It is not necessarily the case that $I_2$ is even with respect to the form
$(\cdot,\cdot)_{\alpha}^{(2)}$, so it might not be the case that this lattice
is isometric to $(L_{2n},J_{2n})$ over $\Z_2$. However, it is the case that both
lattices are isometric under a matrix in $\mathrm{O}_{2n}(\Q_2)$ with coefficients
in $\frac{1}{2}\Z_2$. Thus, the given tuple $(I_1,I_2,\alpha,s)$ yields
an orbit in $\frac{1}{2}V(\Z)$.
\end{itemize}
\end{proof}

\section{The resolvent form}
\label{section:OrbitsA}

We keep the notation $(G,V) = (G_B,V_B)$ from last section.
As before, let $K$ be a field of characteristic zero.
Consider an element $A \in V(K)$ as a $(2n+1) \times 2n$ matrix with entries
in $K$ and associated characteristic polynomial $f(x)$. Then, $A^*A$ is 
a $2n \times 2n$ matrix that is symmetric along the antidiagonal and has
characteristic polynomial $f(x)$. Further, $(g,h) \in \So_{2n+1}\times \So_{2n}$
acts on $A^*A$ by $(g,h) \cdot (A^*A) = hA^*Ah^{*}$.

Let $\PSo_{2n} = \So_{2n}/\mu_2$. Then, we define the representation
$(G_A,V_A) = (\PSo_{2n},\Sym^2(2n))$, where $\Sym^2(2n)$ denotes the $2n \times 2n$
matrices that are symmetric along the antidiagonal, and $G_A$ acts by
conjugation on these matrices. This is (up to a trace zero condition that ultimately does not matter)
the same representation that was studied in \cite{SW}, corresponding to
the Dynkin diagram $A_{2n-1}$. The ring of invariants
of $(G_A,V_A)$ is generated by the coefficients of the characteristic polynomials
of the matrices of $V_A$, and hence we have an isomorphism between $V_A \git G_A$
and $B$. We note, however, that the degrees of the elements of $V_A \git G_A$ are
half the degrees of the corresponding invariants in $B$.

\subsection{Rational orbits}

We recall \cite[Proposition 2.1]{SW}.
\begin{prop}
Let $b \in B^{rs}(K)$ with $\Delta(b) \neq 0$. If the associated characteristic
polynomial is $f(x)$, write $L = K[x]/(f(x))$. Then if $v \in V_{A,b}(K)$,
then $\Stab_{G_A}(v) \cong (\Res_{L/K} \mu_2)_{N \equiv 1}/\mu_2$.
\end{prop}

Therefore, by Proposition \ref{prop:ait}, if $b \in B^{rs}(K)$ the 
$G_A(K)$-orbits in $V_{A,b}(K)$ are in bijection with the pointed kernel of
\[
H^1(K,(\Res_{L/K} \mu_2)_{N \equiv 1}/\mu_2) \to H^1(K,\So_{2n}).
\]
Similarly to $(G,V)$, there is a map $H^1(K,(\Res_{L/K} \mu_2)_{N \equiv 1}/\mu_2) \to (L^{\times}/K^{\times}L^{\times 2})_{N \equiv 1}$
which is bijective or $2$-to-$1$ according to whether the norm map
$N \colon (\Res_{L/K} \mu_2)\mu_2 (K) \to \mu_2$ is surjective or not
(see \cite[Proposition 2.2]{SW} for a more explicit description).

Let $\beta$ denote the image of $x$ inside $L$, so that $L$ has a $K$-basis
$1,\beta,\dots,\beta^{2n-1}$. Given $\alpha \in (L^{\times}/K^{\times}L^{\times 2})_{N \equiv 1}$,
we can define the form $(\cdot,\cdot)_{\alpha} \colon L \times L \to K$ by
\[
(\mu,\lambda)_{\alpha} = \text{coefficient of }\beta^{2n-1}\text{ in }\alpha^{-1}\mu\lambda.
\]
This form has discriminant $1$, up to squares. Then, we have (cf. \cite[Theorem 2.6]{SW}):
\begin{theorem}
There is a bijection between:
\begin{itemize}
\item $\Po(K)$-orbits in $V_{A,b}(K)$; and,
\item elements $\alpha \in (L^{\times}/K^{\times}L^{\times 2})_{N \equiv 1}$
such that $(\cdot,\cdot)_{\alpha}$ is split. 
\end{itemize}
These $\Po(K)$-orbits split into one
or two $\PSo(K)$-orbits according to whether the norm map on $(\Res_{L/K} \mu_2)_{N \equiv 1}/\mu_2$
is surjective or not, respectively.
\end{theorem}

Let $b \in B^{rs}(K)$ correspond to the invariant polynomial $f(x)$,
and consider the hyperelliptic curve $C_b \colon y^2 = xf(x)$ with Jacobian
$J_b = \Jac(C_b)$. We recall the setup of \eqref{eq:JA}:
\[
J_b \xrightarrow{\phi_M} A_b \xrightarrow{\phi} A_b^{\vee} \xrightarrow{} J_b
\]
where $J_b[\phi_M] = M$, $J_b[\phi \circ \phi_M] = M^{\perp}$ and the whole
composition $J_b \to J_b$ is multiplication-by-$2$. In particular, we have
that $A_b[\phi]$ is isomorphic to $M^{\perp}/M$. First, we note the
following fact:
\begin{lemma}
\label{lemma:self-dual}
The isogeny $\phi \colon A_b \to A_b^{\vee}$ is self-dual.
\end{lemma}
\begin{proof}
As suggested by the notation, the abelian varieties $A_b$ and $A_b^{\vee}$
are indeed dual to each other. It is a general fact of principally
polarised abelian varieties that $(J_b/M)^{\vee} \simeq J_b/M^{\perp}$,
which follows from the properties of the Weil pairing.

We will show that, in fact, $\phi$ is a polarisation. Consider the canonical
principal polarisation $\lambda_J \colon J_b \to J_b^{\vee}$ given by the
theta divisor. Then, the associated polarisation $2 \lambda_J$ has kernel 
$J_b[2]$. By \cite[\S 23, Corollary to Theorem 2]{Mumford}, the polarisation
$2 \lambda_J$ descends to a polarisation on $A_b$ (i.e. there is an ample line bundle
$\cL$ on $A_b$ such that its pullback along $\phi_M \colon J_b \to A_b$ is the line bundle associated with $2 \lambda_J$)
if and only if the following two conditions are satisfied:
\begin{itemize}
\item $M$ is contained in the kernel of $2 \lambda_J$; true as $M \leqslant J_b[2]$.
\item $M$ is isotropic with respect to the Weil pairing $e_{2\lambda_J}$;
true as $M \leqslant M^{\perp}$.
\end{itemize}
Therefore, we get a polarisation $\phi_{\cL} \colon A_b \to A_b^{\vee}$.
Unravelling the definitions, $\phi_{\cL}$ satisfies $2 \lambda_J = \phi_M^{\vee} \circ \phi_{\cL} \circ \phi_M$
(viewing $\phi_M^{\vee}$ as a map between $A_b^{\vee}$ and $J_b^{\vee}$),
and the kernel of $\phi_{\cL}$ as an isogeny is isomorphic to $M^{\perp}/M$.
Thus we can identify $\phi = \phi_{\cL}$, and so we are done.
\end{proof}
We also observe the following fact about the stabiliser:
\begin{lemma}
Under the above notation, we have $(\Res_{L/K} \mu_2)_{N \equiv 1}/\mu_2 \cong M^{\perp}/M$.
\end{lemma}
This follows immediately from Proposition \ref{prop:StabJ}. Therefore,
we have a map
\[
A_b^{\vee}(K)/\phi(A_b(K)) \xhookrightarrow{} H^1(K,(\Res_{L/K} \mu_2)_{N \equiv 1}/\mu_2).
\]
\begin{theorem}
\label{theo:comp2}
The composition
\[
A_b^{\vee}(K)/\phi(A_b(K)) \xhookrightarrow{} H^1(K,(\Res_{L/K} \mu_2)_{N \equiv 1}/\mu_2) \xrightarrow{} H^1(K,G_A)
\]
is trivial.
\end{theorem}
\begin{proof}
Note that there's a commutative diagram
\[
\begin{tikzcd}
A_b^{\vee}(K)/\psi(J_b(K)) \arrow[r] \arrow[d] & A_b^{\vee}(K)/\phi(A_b(K)) \arrow[d] \\
H^1(K,J_b[\psi]) \arrow[r] \arrow[d] & H^1(K,A_b[\phi]) \arrow[d] \\
H^1(K,G) \arrow[r]  & H^1(K,G_A) 
\end{tikzcd}
\]
Theorem \ref{theo:comp} shows that the composition in the first column is trivial. The map in
the first row is surjective, and the map in the last row is the surjective
forgetful map $H^1(K,\So_{2n+1}\times \So_{2n}) \to H^1(K,\So_{2n})$.
The result follows.
\end{proof}
Therefore, for all $b \in B^{rs}(K)$ there is a map
\begin{equation}
\label{eq:starSoluble}
A_b^{\vee}(K)/\phi(A_b(K)) \xhookrightarrow{\eta_{A,b}} G_A(K) \backslash V_{A,b}(K),
\end{equation}
and similarly to the last section we will call a $G_A(K)$-orbit in $V_{A,b}(K)$
\emph{$K$-soluble} if it intersects the image of $\eta_{A,b}$. If $K$ is a number
field, we say that an orbit is \emph{locally soluble} if it is $K_v$ soluble
for all completions $K_v$.
The same proof as in Corollary \ref{cor:Selmer} yields:
\begin{cor}
Let $K$ be a number field. Then for $b \in B^{rs}(K)$ we have
\[
\Sel_{\phi}(A_b) \xhookrightarrow{} G_A(K) \backslash V_{A,b}(K).
\]
\end{cor}

In \cite{SW}, a $G_A(K)$-orbit in $V_{A,b}(K)$ is called \emph{reducible} (or distinguished)
if it maps to the element $\alpha = 1$ in 
\[
H^1(K,(\Res_{L/K} \mu_2)_{N \equiv 1}/\mu_2) \to (L^{\times}/K^{\times}L^{\times 2})_{N \equiv 1}.
\]
More precisely, in \cite[\S 2.2]{SW} a distinguished orbit $v_b$ is constructed,
and a $\PSo_{2n}(K)$-orbit is called distinguished if it is $\Po_{2n}(K)$-equivalent
to the constructed orbit $v_b$. This corresponds precisely to the orbits
that map to $1 \in (L^{\times}/K^{\times}L^{\times 2})_{N \equiv 1}$, of
which there are at most two.

\subsection{Integral representatives}

We will prove the equivalent of Theorem \ref{theo:integral}. To do so,
we will use the description of integral orbits in \cite[\S 2.4]{SW}. We
note, however, that there is an oversight in loc. cit. in the case
$p = 2$; the amended statement should read as follows:
\begin{theorem}
\label{theo:intSW}
Let $b \in B^{rs}(\Z_2)$ with invariant polynomial $f(x)$, and let
$L = \Q_2[x]/(f(x))$ and $R = \Z_2[x]/(f(x))$.
There is a bijection between $\mathrm{O}_{2n}(\Z_2)$-orbits in $V_{A,b}(\Z_2)$
and equivalence classes of $(I,\alpha)$, where $\alpha \in L^{\times}$
and $I$ is a fractional ideal of $R$ satisfying $I^2 \subset \alpha R$
and $N(I)^2 = N(\alpha)$, \textbf{which is even with respect to the form $(\cdot,\cdot)_{\alpha}$}. 
Two pairs $(I_1,\alpha_1)$ and $(I_2,\alpha_2)$
are equivalent if there exists $c \in L^{\times}$ such that $I_1 = cI_2$
and $\alpha_1 = c^2 \alpha_2$.
\end{theorem}
\begin{obs}
There's a small convention difference in \cite{SW}, where they take 
$\alpha^{-1}$ where we take $\alpha$.
\end{obs}
The condition of $I$ being even with respect to the form is necessary, and in some cases the constructed
ideals in \cite[\S 2.4]{SW} need not be even. In that case, it can only
be guaranteed that the orbit will fall inside $\frac{1}{2}V_{A,b}(\Z_2)$.

\begin{theorem}
\label{theo:intVstar}
Let $b \in B^{rs}(\Z)$. Every locally soluble orbit in $V_{A,b}(\Q)$ has a
representative in $\frac{1}{2}V_{A,b}(\Z)$.
\end{theorem}
\begin{proof}
As in the proof of Theorem \ref{theo:integral}, it is enough to see that
for all $p$, the map
\[
A_b^{\vee}(\Q_p)/\phi(A_b(\Q_p)) \xhookrightarrow{} G_A(\Q_p) \backslash V_{A,b}(\Q_p)
\]
always intersects $\frac{1}{2}V_{A,b}(\Z_p)$. For $p \neq 2$, this is immediate;
if a $(2n+1) \times 2n$ matrix $A$ has entries in $\Z_p$, then $A^*A$
also does. For $p = 2$, we note that by Theorem \ref{theo:integral} and
Proposition \ref{prop:intOrbits} there exists an ideal called $I_2$ in
there satisfying the hypotheses of Theorem \ref{theo:intSW} with the
corresponding $\alpha$ of the rational orbit.
\end{proof}

\subsection{Proof of Theorem \ref{theo:mainPhi}}
\label{section:mainProof}

We are now in a position to prove Theorem \ref{theo:mainPhi}, 
where we will use Theorem \ref{theo:intVstar}
in conjunction with the counting results of \cite{SW} (which are a particular
case of \cite[\S 8]{LagaThesis}). 
We start by noting the following result from \cite[Lemma 7.1]{LagaPrym}:
\begin{lemma}
\label{lemma:selmerWeights}
Let $K = \R$ or $\Q_p$ for some prime $p$, and write $|\cdot|_K \colon K^{\times} \to \R_{> 0}$
for the normalised absolute value of $K$. Let $A$ be an abelian variety
over $K$ with dual abelian variety $A^{\vee}$, and let $\lambda \colon A \to A^{\vee}$ be
a self-dual isogeny, which will have degree $m^2$ for some $m \in \Z_{\geq 1}$. 
Then the quantity
\[
c(\lambda) := \frac{\# (A^{\vee}(K)/\lambda(A(K)))}{\# A[\lambda](K)}
\]
satisfies $c(\lambda) = 1/|m|_{K}$.
\end{lemma}
In particular, for our map of interest $\phi \colon A_b \to A_b^{\vee}$,
we will have that
\[
c_p(\phi_b) = \begin{cases}
2^{-(n-1)} & \text{if } p = \infty, \\
2^{(n-1)} & \text{if } p = 2, \\
1 & \text{otherwise}.
\end{cases}
\]
We turn our interest to counting $G_A(\Z)$-orbits in $V_A(\Z)$. We will
do so by imposing infinitely many congruence conditions:
\begin{defi}
\label{defi:acceptable}
A map $w \colon V_A(\Z) \to [0,1]$ is said to be \emph{defined by infinitely many congruence conditions}
if for each prime $p$ there exist functions $w_p \colon V_A(\Z_p) \to [0,1]$ such that
\begin{itemize}
\item $w_p$ is $G_A(\Z_p)$-invariant;
\item $w_p$ is locally constant outside the subset $\{v \in V_A(\Z_p) \mid \Delta(v) = 0\}$;
\end{itemize}
which satisfy $w = \prod_p w_p$. We additionally say that
$w$ is \emph{acceptable} if the product
\[
\prod_{p} \int_{v \in V_A(\Z_p)} w(v) dv
\]
does not diverge to zero, where $dv$ is normalised so that the volume
of $V_A(\Z_p)$ is equal to $1$ for all $p$.
\end{defi}
\begin{obs}
The acceptability condition is achieved in many instances in the literature
by guaranteeing that $1-\int_{v \in V_A(\Z_p)} w(v) dv = O(p^{-2})$
for large enough $p$.
\end{obs}

Consider a $G_A(\Z)$-invariant subset $A \subset V_A(\Z)$, and let 
$w\colon V_A(\Z) \to \R$ be an acceptable function defined by infinitely many 
congruence conditions. We denote
\[
N_w^*(A,X) = \sum_{v \in G_A(\Z) \backslash A_{<X}} \frac{w(v)}{\# \Stab_{G_A(\Z)}(v)}.
\]
Recall that a $G_A(\R)$-orbit in $V_{A,b}(\R)$ is called $\R$-soluble if it
falls in the image of $\eta_{A,b}$ in \eqref{eq:starSoluble}. Observe that
the number of $G_A(\R)$-soluble orbits in $V_{A,b}(\R)$ is $\# A_b^{\vee}(\R)/\phi(A_b(\R))$.
Then, analogously to \cite[Theorem 8.18]{LagaThesis} we get that
\[
N_w^*(A,X) \leq \lp \prod_{p} \int_{v \in V_A(\Z_p)} w(v) dv \rp \frac{|W_1|}{2^{n-1}} \vol(G_A(\Z)\backslash G_A(\R))\vol(B(\R)_{<X}) + o(X^{\dim V_A}),
\]
where $W_1 \in \Q^{\times}$ is a fixed scalar number, and where $w = \prod_p w_p$
are the congruence conditions defining $w$. 

To estimate the size of $\Sel_{\phi}(A_b)$,
we note that the non-trivial torsion point $T_b \in A_b^{\vee}[\hat{\psi}]$
generates a subgroup $S_T$ in $\Sel_{\phi}(A_b)$ of order dividing $2$.
In the map
\[
\Sel_{\phi}(A_b) \xhookrightarrow{} G_A(K) \backslash V_{A,b}(K),
\]
the elements of $\Sel_{\phi}(A_b)$ which intersect the reducible orbits
correspond exactly to the subgroup $S_T$, and the complement of $S_T$
falls entirely in the irreducible orbits. Given that $\# S_T \leq 2$, it
suffices to bound $\Sel_{\phi}(A_b) \setminus S_T$ by looking at irreducible
orbits.

We can prove our results in higher generality by imposing congruence
conditions on $B$. We say that a set $\cB \subset B(\Z)^{rs}$ is \emph{defined by finitely many congruence conditions}
if it is the preimage of the reduction map $B(\Z)^{rs} \to B(\Z/N\Z)$
for some $N \geq 1$. We will prove the following:
\begin{theorem}
\label{theo:mainPhiCong}
Let $\cB \subset B(\Z)$ be defined by finitely many congruence conditions.
Then
\[
\lim_{X \to \infty} \frac{\sum_{b \in \cB, \, \Ht(b) < X} \# (\Sel_{\phi}(A_b) \setminus S_T)}{\#\{b \in \cB \mid \Ht(b) < X\}} \leq 4.
\] 
\end{theorem}
\begin{proof}
Let $\cB_p$ denote the closure of $\cB$ inside $B^{rs}(\Z_p)$. For our counting
result, it will suffice to count those irreducible $G_A(\Z)$-orbits in
$V_A(\Z)$ corresponding to Selmer elements, as given by Theorem \ref{theo:intVstar}.
Given that we are only guaranteed to have orbits in $\frac{1}{2}V_A(\Z)$,
and that $\Sel_{\phi}(A_b) \simeq \Sel_{\phi}(A_{\lambda\cdot b})$ for
any $\lambda \in \Q^{\times}$, it will suffice to look at orbits with
invariants in $2 \cdot \cB$, for which Selmer elements will always have
integral representatives.
We choose the counting function
\[
w(v) = \begin{cases}
\lp \sum_{v' \in G_A(\Z) \backslash(G_A(\Q) \cdot v \cap V_A(\Z))} \frac{\# \Stab_{G_A(\Q)}(v')}{\# \Stab_{G_A(\Z)}(v')}\rp^{-1} & \text{if } \pi(v) \in 2 \cdot \cB \text{ and } v \text{ is locally soluble,} \\
0 & \text{otherwise.}
\end{cases}
\]
This is defined by congruence conditions by the functions
\[
w_p(v) = \begin{cases}
\lp \sum_{v' \in G_A(\Z_p) \backslash(G_A(\Q_p) \cdot v \cap V_A(\Z_p))} \frac{\# \Stab_{G_A(\Q_p)}(v')}{\# \Stab_{G_A(\Z_p)}(v')}\rp^{-1} & \text{if } \pi(v) \in 2 \cdot \cB_p \text{ and } v \text{ is soluble,} \\
0 & \text{otherwise,}
\end{cases}
\]
by an analogous argument to \cite[Proposition 3.6]{BSquartics}. The last
part of \cite[Lemma 8.5]{LagaThesis} gives
\begin{align*}
\int_{v \in V_A(\Z_p)} w(v) dv &= |W_1|_p \vol(G_A(\Z_p)) \int_{b \in 2\cdot \cB_p} \frac{\# A_b^{\vee}(\Q_p)/\phi(A_b(\Q_p))}{\# A_b[\phi](\Q_p)}db \\
&= |W_1|_p \vol(G_A(\Z_p)) |2^{-(n-1)}|_p |2^{n(2n+1)}|_p\vol(\cB_p),
\end{align*}
using Lemma \ref{lemma:selmerWeights} in the last line. An explicit
computation on $\vol(G_A(\Z_p))$ shows that is it $1 - O(p^{-2})$, where
the implicit constant is independent of $p$. Therefore, $w$ is acceptable
in the sense of Definition \ref{defi:acceptable}.
Under this counting function, we have that for any given locally soluble
$v \in V_A(\Z)$ with $\pi(v) \in \cB$:
\[
\sum_{v' \in G_A(\Q)v \cap V_A(\Z)} \frac{w(v')}{\# \Stab_{G_A(\Z)}(v')} = \frac{1}{\# \Stab_{G_A(\Q)}(v)}.
\]
$100\%$ of the time, this quantity is equal to $1$ by \cite[Proposition 23]{SW}.
Thus, we have that
\[
\sum_{b \in \cB_{<X}} \#(\Sel_{\phi}(A_b)\setminus S_T) = N_w^*(V_A(\Z)^{irr} \cap V_A(\R)^{sol},2X) + o(X^{n(2n+1)}).
\]
With an elementary point-counting argument, we can see that
\[
\lim_{X \to \infty} \frac{\prod_p \vol(\cB_p) \vol(B(\Z)_{<2X})}{\#\{b \in \cB \mid \Ht(b) < X\}} = 2^{n(2n-1)}.
\]
Putting it all together, we have that
\[ 
\frac{N_w^*(V_A(\Z)^{irr} \cap V_A(\R)^{sol},2X)}{\#\{b\in \cB \mid \Ht(b) < X\}} \leq \vol(G_A(\Z)\backslash G_A(\R)) \prod_p \vol(G_A(\Z_p)),
\]
which equals the Tamagawa number of $G_A = \PSo_{2n}$, which is $4$.
This concludes the proof.
\end{proof}
Theorem \ref{theo:mainPhi} then follows from Theorem \ref{theo:mainPhiCong},
since $S_T$ has size at most $2$.

\section{Orbit parametrisations for $C_{2n}$}
\label{section:OrbitsC}

This section is equivalent to Section \ref{section:OrbitsB}, but now in
the $C_{2n}$ case. For this section, let us denote $(G,V) := (G_C,V_C)$
and fix a field $K$ of characteristic zero.

\subsection{Rational orbits}

Similarly to the $B_{2n}$ case, we will start by constructing a distinguished
orbit, from which we will obtain the other ones.

\subsubsection*{The distinguished orbit}

Let $b = (p_2,\dots,p_{4n}) \in B(K)$, and suppose that $\Delta(b) \neq 0$.
Consider the polynomial $f(x) = x^{2n}+p_2x^{2n-1}+\dots+p_{4n}$, which is separable
by hypothesis. We consider the étale algebras
\[
M = \frac{K[x]}{(f(x^2))} = K[\beta], \quad L = \frac{K[x]}{(f(x))} = K[\gamma].
\]
We note that $M = K\langle 1,\beta,\dots,\beta^{4n-1} \rangle$, and that
there is a decomposition as $K$-vector spaces $M = L \oplus \beta L$,
corresponding to the even and odd-degree part.

Consider the involution $\eps \colon M \to M$ given by $\beta \mapsto -\beta$,
and let 
\begin{align*}
(\cdot,\cdot) \colon M \times M & \to K \\
(\mu,\lambda) & \mapsto \text{ coefficient of }\beta^{4n-1}\text{ in } \lambda \eps(\mu).
\end{align*}
Note that $(L,L) = (\beta L, \beta L) = 0$. The Gram matrix of this form
can be written as $G = \begin{pmatrix} 0 & G_1 \\ G_2 & 0\end{pmatrix}$,
where $G_1 = -{}^tG_2$. Then, there exists a change-of-basis matrix of the
form $S = \begin{pmatrix} S_1 & 0 \\ 0 & S_2\end{pmatrix}$ such that
$S^t G S = \begin{pmatrix} 0 & J_{2n} \\ -J_{2n} & 0\end{pmatrix}$ and
$\det(S_1) = \det(S_2) = 1$ (e.g. by taking $S_1 = J_{2n}G_1^{-1}$, $S_2 = \id$).

The map $T_{\beta} \colon M \to M$ given by multiplication by $\beta$ is
anti-self-adjoint with respect to the form $(\cdot,\cdot)$. With respect
to the above basis, the matrix of $T_{\beta}$ is an element of $V(K)$.

\begin{lemma}
Let $v \in V^{rs}(K)$. Then $\Stab_{G}(v) \cong (\Res_{L/K} \mu_2)/\mu_2$.
\end{lemma}
\begin{proof}
Given that $v$ is regular semisimple, the centraliser of $T_{\beta}$
in $\GL(M)$ is $M^{\times}$. An element of the centraliser has to respect
the $L$ and $\beta L$ parts, so it actually has to lie in $L^{\times}$.
The condition that $\lambda \in L^{\times}$ has to respect the symplectic
form translates to $\lambda^2 = 1$.
Therefore, elements of the stabiliser lie in $\Res_{L/K} \mu_2$. The condition
that $G = \GL_{2n}/\mu_2$ (as opposed to $\GL_{2n}$) introduces the $\mu_2$ quotient
in the statement.
\end{proof}

\subsubsection*{The other orbits}

By Proposition \ref{prop:ait}, any other rational orbits that are $G(K)$-conjugate
to the constructed rational orbit are in bijection with elements in the
kernel of the map $H^1(K,\Stab_G(v)) \to H^1(K,G)$. By looking at the
commutative diagram
\[
\begin{tikzcd}
H^1(K,\mu_2) \arrow[r] \arrow[d] & H^1(K,\Res_{L/K} \mu_2) \arrow[r] \arrow[d] & H^1(K,(\Res_{L/K} \mu_2)/\mu_2)\arrow[r] \arrow[d] & H^2(K,\mu_2) \arrow[d]\\
H^1(K,\mu_2) \arrow[r] & H^1(K,\GL_{2n}) \arrow[r] & H^1(K,\GL_{2n}/\mu_2)\arrow[r] & H^2(K,\mu_2),
\end{tikzcd}
\]
any element in the kernel of $H^1(K,\Stab_G(v)) \to H^1(K,G)$ maps to
the trivial element in $H^2(K,\mu_2)$, and therefore comes from an element
of $H^1(K,\Res_{L/K} \mu_2) \cong L^{\times}/L^{\times 2}$ up to
an element of $H^1(K,\mu_2) \cong K^{\times}/K^{\times 2}$. Moreover,
any element coming from $L^{\times}/(K^{\times} L^{\times 2})$ must map
to the trivial element in $H^1(K,\GL_{2n}/\mu_2)$, as $H^1(K,\GL_{2n}) = \{1\}$
by Hilbert's Theorem 90. Therefore, we conclude that the rational orbits
of given invariant are in bijection with the set $L^{\times}/(K^{\times} L^{\times 2})$.

Given an element $\alpha \in L^{\times}/(K^{\times} L^{\times})$, 
consider the form
\begin{align*}
(\cdot,\cdot) \colon M \times M & \to K \\
(\mu,\lambda) & \mapsto \text{ coefficient of }\beta^{4n-1}\text{ in } \alpha^{-1}\lambda \eps(\mu).
\end{align*}
Then, a procedure analogous to the one in Section \ref{section:OrbitsB} 
gives rise to a distinct rational orbit.

Given $b \in B(K)$, let $C_b \colon y^2 = xf(x)$ be the corresponding
hyperelliptic curve and $J_b = \Jac(C_b)$. Recall the following composition
of isogenies
\[
J_b \xrightarrow{\phi_M} A_b \xrightarrow{\phi} A_b^{\vee} \xrightarrow{\phi_M^{\vee}} J_b.
\]
Denote $\psi^{\vee} = \phi_M^{\vee} \circ \phi$.
By construction, we get:
\begin{prop}
We have an isomorphism $A_b[\psi^{\vee}] \cong (\Res_{L/K} \mu_2)/\mu_2$.
\end{prop}

Consequently, we get an inclusion $J_b(F) / \phi^{\vee}(A_b(F)) \xhookrightarrow{} H^1(K,A_b[\psi^{\vee}])$.

\begin{theorem}
The composition
\[
J_b(F) / \psi^{\vee}(A_b(F)) \xhookrightarrow{} H^1(K,A_b[\psi^{\vee}]) \to H^1(K,\GL_{2n}/\mu_2)
\] 
is trivial.
\end{theorem}
\begin{proof}
If we look at $\GL_{2n}$-orbits, the stabiliser in that case is isomorphic
to $\Res_{L/K} \mu_2 \cong J_b[2]$, and we have a commutative diagram
\[
\begin{tikzcd}
J_b(F) / 2J_b(F) \arrow[r] \arrow[d] &  J_b(F)/\psi^{\vee}(A_b(F)) \arrow[d] \\
H^1(K,J_b[2]) \arrow[r] \arrow[d] &  H^1(K,A_b[\psi^{\vee}]) \arrow[d] \\
H^1(K,\GL_{2n}) \arrow[r] & H^1(K,G) \\
\end{tikzcd}
\]
Because $H^1(K,\GL_{2n}) = \{1\}$, and because the top map is surjective,
we conclude that the composition in the rightmost column is trivial.
\end{proof}

Therefore, we get an inclusion
\begin{equation}
\label{JacC}
J_b(F) / \psi^{\vee}(A_b(F)) \xhookrightarrow{} G(F) \backslash V_b(F).
\end{equation}

By a similar argument as in Corollary \ref{cor:Selmer}, we get:

\begin{cor}
We have an inclusion
\[
\Sel_{\psi^{\vee}}A_b \xhookrightarrow{} G(\Q) \backslash V_b(\Q).
\]
\end{cor}
\begin{proof}
It suffices to see that the map $H^1(\Q,G) \to \prod_{p}H^1(\Q_p,G)$
has a trivial pointed kernel. In this situation, we can't directly apply
\cite[Proposition 6.8]{LagaThesis} as $G$ is not semisimple, so we 
instead do a direct proof. The exact sequence
\[
\begin{tikzcd}
1 \arrow[r] & \mu_2 \arrow[r] & \GL_{2n} \arrow[r] & G \arrow[r] & 1
\end{tikzcd}
\]
gives the following commutative diagram with exact rows:
\[
\begin{tikzcd}
H^1(\Q,\GL_{2n}) \arrow[r] \arrow[d] & H^1(\Q,G) \arrow[r] \arrow[d] & H^2(\Q,\mu_2) \arrow[d] \\
\prod_{p} H^1(\Q_p,\GL_{2n}) \arrow[r] & \prod_{p} H^1(\Q_p,G) \arrow[r] & \prod_{p} H^2(\Q_p,\mu_2) \\
\end{tikzcd}
\]
By Hilbert's Theorem 90, $H^1(\Q,\GL_{2n})$ and $H^1(\Q_p,\GL_{2n})$ are
trivial. Moreover, the Albert--Brauer--Hasse--Noether theorem implies 
that the map $H^2(\Q,\mu_2) \to \prod_{p} H^2(\Q_p,\mu_2)$
is injective. Thus, if an element $c \in H^1(\Q,G)$ maps to the trivial
element of $\prod_{p} H^1(\Q_p,G)$, then it also maps to the trivial element
of $\prod_{p} H^2(\Q_p,\mu_2)$, and the injectivity of both maps $H^1(\Q,G) \to H^2(\Q,\mu_2) \to \prod_{p} H^2(\Q_p,\mu_2)$
means that $c$ is trivial, as desired.
\end{proof}

\subsection{Integral orbits}

We wish to show that any rational orbit corresponding to a Selmer element
has integral representatives. In order to do this, first we parametrise
what the integral representatives can be.

Define $R =  \Z_p[\gamma] = \Z_p\cdot\la 1,\gamma,\dots,\gamma^{2n-1}\ra$,
an order inside $L$.

\begin{prop}
\label{prop:intC}
Let $b \in B(\Z_p)$. The set of $G(\Z_p)$-orbits in $V_b(\Z_p)$ is in bijection
with equivalence classes of tuples $(I_1,I_2,\alpha)$, where $\alpha \in L^{\times}$
and $I_1,I_2$ are fractional ideals of $R$ satisfying:
\begin{enumerate}
\item $I_1 I_2 \subset \alpha R$;
\item $N(I_1)N(I_2) = N(\alpha)$;
\item $I_1 \subset I_2 \subset \gamma^{-1}I_1$;
\end{enumerate}
Two tuples $(I_1,I_2,\alpha)$, $(I_1',I_2',\alpha')$ are equivalent if there exists 
$c \in L^{\times}$ such that $(I_1,I_2,\alpha) = (cI_1',cI_2',c^2\alpha')$.
\end{prop}
\begin{proof}
Let $I = I_1 \oplus \beta I_2$ be a lattice inside $M$. If we consider
$M$ with basis given by any given $\Z$-bases of $I_1,I_2$, the operator $T_{\beta}$
is an integral operator (by hypothesis 3) with associated matrices $(A,C)$.
We now need to show that there exist
$\Z$-bases of $I_1,I_2$ such that $A = A^*$ and $C = C^*$.

Using last section's notation, we see that the form $(\cdot,\cdot)_{\alpha}$ has
an integral Gram matrix (by hypothesis 1), and that the associate matrices
$(G_1,G_2)$ have the same determinant $\pm 1$ (by hypothesis 2). Then,
it is safe to choose change-of-bases matrices $(S_1,S_2)$ with same
determinant $\pm 1$ as appropriate to make $A$ and $C$ self-adjoint with respect
to the usual inner product.

Conversely, given two matrices $(A,C) \in V_b(\Z)$, we construct $I_1,I_2$
as follows. First, construct $R = \Z[x]/g(x) = \Z[\gamma]$, and let $I_2 \cong \Z^n$
as abelian groups. To make $I_2$ into an $R$-module, we need to specify
how $\gamma$ acts: by viewing elements of $I_2$ as column vectors, we 
let $\gamma$ act on $I_2$ as the matrix $CA$ (this action is well-defined,
because $g(CA) = 0$ by the Cayley--Hamilton theorem). We then let $I_1$ be the
submodule $C\cdot I_2$, and obtain $\alpha$ from the rational orbit of
$(A,C)$. If we choose $\cB_1$ and $\cB_2$ to be $\Z$-bases of $I_1$ and
$I_2$, then the matrix of the operator $T_{\beta}$ on the basis of $M$ 
given by $\cB_1 \cup \beta \cB_2$ is the one corresponding to $(A,C)$.
The fact that the lattice $I_1\oplus\beta I_2$ of $M$ is self-dual to the form $(\cdot,\cdot)_{\alpha}$
translates to the hypotheses $I_1I_2 \subset \alpha R$ and $N(I_1)N(I_2) = N(\alpha)$.

Therefore, the two constructions $(A,C) \mapsto (I_1,I_2,\alpha)$ and
$(I_1,I_2,\alpha) \mapsto (A,C)$ are inverse to each other, as desired.
\end{proof}

\begin{theorem}
The image of the map
\[
J_b(\Q_p) / \psi^{\vee}(A_b(\Q_p)) \xhookrightarrow{} G(\Q_p) \backslash V_b(\Q_p).
\]
always intersects $V_b(\Z_p)$.
\end{theorem}
\begin{proof}
Let $L_1 = \frac{K[x]}{(xf(x))} = K[\beta_1]$, which can be decomposed
as $L_1 \cong \Q_p \times L$. By \cite[Lemma 4.9]{LTA2n},
there exists a fractional ideal $\tilde{I_1}$ of the order $R_1 = \Z_p[\beta_1]$
satisfying $\tilde{I_1}^2 \subset \tilde{\alpha}R_1$ (where $\tilde{\alpha}$
is a lifting of $\alpha$ to $L_1$) and $N(\tilde{I_1})^2 = N(\alpha)$,
where the norms are taken with respect to $R_1$.

Take $I_1$ to be the reduction of $\tilde{I_1}$ to $R$. For an element
$\lambda \in L_2$, define $\tilde{\lambda} \in L_1$ to be $\tilde{\lambda} = (0,\lambda) \in \Q_p \times L$.
Then, similarly to the proof of Theorem \ref{theo:integral}, consider
\[
I_1' = \{\lambda \in L_2 \mid \beta \tilde{\lambda} \in I_1\}.
\]
This is a fractional ideal of $R$. We now claim that taking $I_2 = I_1'$
satisfies the conditions of Proposition \ref{prop:intC}. Indeed, by the
same proof as Theorem \ref{theo:integral}, the ideals $I_1$ and $I_2$
are dual to each other with respect to $(\cdot,\cdot)_{\alpha}$, so this
guarantees the first two points of the proposition. The final point
is an easy computation.
\end{proof}

\section{Counting orbits for $C_{2n}$}
\label{section:countingC}

Fix $(G,V) = (G_C,V_C)$.
In this section, we develop the necessary geometry-of-numbers methods
we need to count $G(\Z)$-orbits in $V(\Z)$. We note that 
$G = \GL_{2n}/\mu_2$ is not semisimple, a fact that will introduce some
important technical differences with respect to similar arguments found in
the literature.

\subsection{Measures}

Recall that $B \cong \Spec \Z[b_2,\dots,b_{4n}]$.
We start by recording a useful numerological fact.

\begin{lemma}
\label{lemma:sum}
$2+4+\dots+4n = \dim_{\Q} V$.
\end{lemma}

To obtain a measure on $G = \GL_{2n}/\mu_2$, we start by explicitly
writing down the Iwasawa decomposition of $G(\R)$. We have that
$G(\R) = \Lambda N T K$, where $\Lambda = \{\lambda I_{2n} \mid \lambda > 0\}$,
$N$ consists of unipotent lower triangular matrices, $T = \{\diag(t_1,\dots,t_{2n}) \mid t_i > 0, \, t_1\dots t_{2n} = 1\}$
and $K$ is a maximal compact subgroup isomorphic to $\textrm{O}_{2n}/\mu_2$.
Then, it follows that the map
\[
\Lambda \times N(\R) \times T(\R) \times K \to G(\R)
\]
is a diffeomorphism by an argument similar to \cite[Chapter 3, \S 1]{SL2}.
We can explicitly determine a Haar measure for $G$ in terms of $\Lambda,N,T,K$:

\begin{prop}
A Haar measure $dg$ on $G(\R)$ can be given by $dg = \delta^{-1}(t) \, d^{\times}\lambda \, dn \, d^{\times}t \, dk$,
where $d^{\times}{\lambda} = \frac{d\lambda}{\lambda}$, $dn$ and $dk$ are Haar
measures on $N(\R)$ and $K$ respectively, $d^{\times}t = \prod_{i=1}^{2n-1} \frac{dt_i}{t_i}$
and $\delta^{-1}(t) = \prod_{i = 1}^{2n-1} t_i^{4n-2i}$.
\end{prop}
If $T = \diag(t_1,\dots,t_{2n})$, write the following change of variables
for $1 \leq m \leq 2n$:
\[
t_m = \prod_{k=1}^{m-1} s_k^{-k}\prod_{k=m}^{2n-1} s_k^{2n-k}.
\]
The conditions that $t_k/t_{k+1} < c$ for all $k \in \{1,\dots,2n-1\}$
translate to $s_k < c'$ for all $k \in \{1,\dots,2n-1\}$ and some $c' > 0$.

Using this decomposition, we can obtain a fundamental domain for the
action of $G(\Z)$ on $G(\R)$. In fact, we can choose it to be \emph{box-shaped at infinity};
a notion that we now explain.
Given a positive constant $c > 0$,
define $T_c = \{ \diag(t_1,\dots,t_{2n}) \in T(\R) \mid t_2/t_1 < c, \dots, t_{2n}/t_{2n-1} < c \}$.
Then, define a \emph{Siegel set} to be a set of the form $\cS = \Lambda \omega T_c K$,
where $\omega$ is a compact subset of $N(\R)$ and $c > 0$. Then, we say
that a set $\cF \subset G(\R)$ is \emph{box-shaped at infinity} if there
exist Siegel subsets $\cS_1,\cS_2$ with $\cS_1 \subset \cF \subset \cS_2$
and satisfying:
\begin{itemize}
\item There exists an open subset $\cU_1 \subset \cS_1$ of full measure
such that every $G(\Z)$-orbit in $G(\R)$ intersects $\cU_1$ at most once.
\item Every $G(\Z)$-orbit in $G(\R)$ intersects $\cU_2$ at least once.
\item There exists a sufficiently small $c$ such that $\cS_1 \cap \Lambda NT_cK = \cS_2 \cap \Lambda NT_cK$.
\end{itemize}
\begin{prop}
There exists a box-shaped fundamental domain for the action of $G(\Z)$
on $G(\R)$.
\end{prop}
\begin{proof}
For $G' = \SL_{2n}/\mu_2$, the proof of \cite[\S 4.3]{Oller} shows that
there exists a box-shaped fundamental domain $\cS_1' \subset \cF' \subset \cS_2'$ for $G'$ (for which
the definition is the same as above, removing $\Lambda$). Considering
$\cF := \Lambda \cdot \cF'$, $\cS_1 := \Lambda \cdot \cS_1'$ and $\cS_2 := \Lambda \cdot \cS_2'$
works.
\end{proof}

Within this set-up, we have the following change-of-measure formula:
\begin{lemma}
\begin{enumerate}
\item Let $p$ be a prime, and let $m_p \colon V^{rs}(\Z_p) \to \R_{\geq 0}$
be defined as
\[
m_p(v) = \sum_{v' \in G(\Z_p)\backslash (G(\Q_p)\cdot v \cap V(\Z_p))} \frac{\# \Stab_{G(\Q_p)}(v')}{\# \Stab_{G(\Z_p)}(v')}.
\]
Then $m_p$ is locally constant.
\item Let $\psi_p \colon V^{rs}(\Z_p) \to \R_{\geq 0}$ be a bounded, locally
constant function satisfying $\psi_p(v) = \psi_p(v')$ whenever $v,v' \in V(\Z_p)$
are $G(\Q_p)$-conjugate. Then there exists $W_0 \in \Q^{\times}$, independent
of $p$, such that
\[
\int_{v \in V^{rs}(\Z_p)} \psi_p(v) dv = |W_0| \vol(G(\Z_p)) \int_{b \in B^{rs}(\Z_p)} \sum_{v \in G(\Q_p) \backslash V_b(\Z_p)} \frac{m_p(v)\psi_p(v)}{\# \Stab_{G(\Q_p)}(v)}.
\]
\end{enumerate}
\end{lemma}
\begin{proof}
This is analogous to \cite[Proposition 3.3]{RTE78}, using Lemma \ref{lemma:sum}.
\end{proof}

We can also construct special subsets of $V^{rs}(\R)$ to serve as fundamental
domains of the action of $G(\R)$ on $V^{rs}(\R)$. Similarly to \cite[\S 2.9]{ThorneE6},
we can find open subsets $L_1,\dots,L_k$ of $\{b \in B^{rs}(\R) \mid \Ht(b) = 1\}$
together with sections $s_i \colon L_i \to V(\R)$ of the invariant map
$\pi \colon V \to B$ satisfying the following properties:
\begin{itemize}
\item For each $i$, the set $L_i$ is connected and semialgebraic, and
$s_i$ is a semialgebraic map with bounded image.
\item Let $D = \R_{> 0}$. We then have an equality:
\[
V^{rs}(\R) = \bigcup_{i=1}^k G(\R) \cdot  D \cdot s_i(L_i).
\]
\end{itemize}

\subsection{Averaging and reductions}

We now carry out the core of our geometry-of-numbers arguments to count
$G(\Z)$-orbits in $V(\Z)$. Many of the steps are completely analogous
to the case where $G$ is semisimple, and will be duly omitted. The reader
can check \cite[\S 2.3]{BSquartics} for further context.

Let $A \subset V(\Z)$ be a $G(\Z)$-invariant set, and denote
\[
N(A,X) = \sum_{v \in G(\Z) \backslash A_{< X}} \frac{1}{\# \Stab_{G(\Z)}(v)}.
\]
Let $F$ be a field of characteristic zero. We say that an element $v \in V(F)$
with $b = \pi(v)$ is:
\begin{itemize}
\item \emph{$F$-reducible} if $\Delta(b) = 0$ or if there exists an element
$w = (A,C)$ in the $G(\Q)$-orbit of $v$ such that either $A = J_{2n}$
or $C = J_{2n}$; and \emph{$F$-irreducible} otherwise. 
\item $F$-soluble if $\Delta(b) \neq 0$ and $v$ lies in the image of
\[
J_b(F)/\psi^{\vee}(A_b^{\vee}(F)) \xhookrightarrow{} G(F) \backslash V_b(F)
\]
of \eqref{JacC}.
\end{itemize}
Given $A \subset V(\Z)$, we will denote by $A^{irr}$ its set of $\Q$-irreducible
elements, and we will also denote by $V(\R)^{sol}$ the set of $\R$-soluble
elements of $V(\R)$.

\begin{theorem}
\label{theo:count1}
There exists a constant $C > 0$ such that
\[
N(V(\Z)^{irr} \cap V(\R)^{sol}, X) = C X^{\dim V} \log X + O(X^{\dim V})
\]
as $X \to \infty$.
\end{theorem}
Similarly to \cite[Theorem 8.8, Proposition 8.10]{LagaThesis}, it suffices to prove the following result. Let $I$
be a subset of $\{1,\dots,k\}$, and denote $L_I = \cap_{i \in I} \pi (G(\R) \cdot s_i(L_i))$.
Also denote $s_I$ to be the restriction of some $s_i$ for $i \in I$ (which
might depend on the choice of $i$, but ultimately that choice does not
matter). The following theorem implies Theorem \ref{theo:count1}.

\begin{theorem}
\label{theo:count2}
Suppose $I$ is a non-empty subset of $\{1,\dots,k\}$ and let $(L,s) = (L_I,s_I)$.
Then there exists a constant $C_I$ such that
\[
N(V(\Z)^{irr} \cap V(\R)^{sol}, X) = C_I X^{\dim V} \log X + O(X^{\dim V})
\]
as $X \to \infty$.
\end{theorem}

From now on, let us fix $(L,s) = (L_I,s_I)$ for some choice of $I$.
We now carry out the main steps of the proof of Theorem \ref{theo:count2}
(and therefore of Theorem \ref{theo:count1}). We begin with an averaging
trick, which works just the same as in \cite[\S 2.3]{BSquartics}. Fix $G_0 \subset G(\R) \times \R_{> 0}$,
a compact, semialgebraic subset of non-empty interior satisfying $KG_0 = G_0$,
$\vol(G_0) = 1$ and $G_0 = G_0' \times [1,K_0]$ for some $G_0' \subset G(\R)$
and $K_0 > 1$.
Let $A \subset V(\Z) \cap G(\R) \cdot D \cdot s(L)$
be a $G(\Z)$-invariant set. Then
\begin{equation}
\label{N(A,X)}
N(A,X) = \frac{1}{r} \int_{d > 0} \int_{g \in \cF} \# [A \cap (g d \cdot s(L))_{< X}] \, dg \, d^{\times}d.
\end{equation}
If we take a set $A \subset V(\Z) \cap G(\R) \cdot D \cdot s(L)$ that is
not necessarily $G(\Z)$-invariant, we \emph{define} $N(A,X)$ to be the
expression in \eqref{N(A,X)}.
We will need some reductions. Let us denote:
\begin{itemize}
\item $V(\Z)^{cusp}$ to be the \emph{cuspidal region}, which is the subset
of $(A,C) \in V(\Z)$ such that $a_{1,2n} = 0$ or $c_{1,2n} = 0$.
\item $V(\Z)^{main} = V(\Z) \setminus V(\Z)^{cusp}$.
\item $V(\Z)^{bigstab}$ to be the subset of elements $v\in V(\Z)$ such
that $\# \Stab_{G(\Q)}(v) > 1$.
\end{itemize}
\begin{prop}
\label{prop:cusp3}
The following hold:
\begin{enumerate}
\item $N(V(\Z)^{cusp} \cap V(\Z)^{irr}, X) = O(X^{\dim V})$
\item $N(V(\Z)^{main} \cap V(\Z)^{red}, X) = o(X^{\dim V}\log X)$.
\item $N(V(\Z)^{bigstab} \cap V(\Z)^{irr}, X) = o(X^{\dim V}\log X)$.
\end{enumerate}
\end{prop}

We will prove the first item in Section \ref{subs:cusp}, while the
second and third item follow from the same proof as in \cite[Propositions 8.16 and 8.21]{LagaThesis}.

The consequence of these reductions is that to count irreducible elements,
it suffices to count elements in the main body of the representation.
This will be done using geometry-of-numbers, and more specifically the
following version of Davenport's lemma, due to Barroero--Widmer \cite[Theorem 1.3]{BWDav}:
\begin{lemma}
\label{lemma:davenport}
Let $m,n \geq 1$ be integers and let $Z \subset \R^{m+n}$ be a semialgebraic
subset. For $T \in \R^{m}$, let $Z_T = \{x \in \R^n \mid (T,x) \in Z\}$,
and suppose that all such sets $Z_T$ are bounded. Then for any unipotent
upper-triangular matrix $u \in \GL_n(\R)$, we have
\[
\#(Z_T \cap u\Z^n) = \vol(Z_T) + O(\max\{1,Z_{T,j}\}),
\]
where $Z_{T,j}$ runs over all orthogonal projections of $Z_T$ to all 
$j$-dimensional coordinate subspaces, ($1 \leq j \leq n-1$). Moreover,
the implied constant depends only on $Z$.
\end{lemma}
We can then estimate the number of points inside the main body. This
will be done similarly to \cite[Proposition 8.15]{LagaThesis}, with a 
key difference introduced by $G$ not being semisimple.
\begin{prop}
\label{prop:mainBody}
Let $A = V(\Z)^{main} \cap G(\R)\cdot D \cdot s(L)$.
There exists a constant $C > 0$ such that
$N(A,X) = CX^{\dim V}\log X + O(X^{\dim V})$.
\end{prop}
\begin{proof}
In the main body, an element $v = (A,C) \in V(\Z)^{main}$ has that
$a_{1,2n} \neq 0$ and $c_{1,2n} \neq 0$. Given that the space $\omega \cdot G_0 \cdot s(L)$
is bounded, there exists a constant $J > 0$ such that $|a_{i,j}|,|c_{i,j}| \leq J$
for all $(A,C) \in \omega \cdot G_0 \cdot s(L)$. Then, the condition that
$a_{1,2n} \neq 0$ and $c_{1,2n} \neq 0$ for an element $v = (A,C) \in (d\lambda n t \cdot s(L))_{< X}$
translates to asking that $d\lambda^2 t_1^2 \leq 1/J$ and
$d\lambda^{-2} t_{2n}^{-2} \leq 1/J$. Note that $t_1 = \prod_{k=1}^{2n-1} s_k^{2n-k} \gg 1$
and $t_{2n} = \prod_{k=1}^{2n-1} s_k^{-k} \ll 1$, so these two last conditions
imply that $d^{-1} \ll \lambda^2 \ll d$.

We wish to apply Davenport's lemma (Lemma \ref{lemma:davenport}). First, we estimate the size of the
error term. Let $\Phi_V$ be the set of weights of $V$ under the action
of $T$, which correspond to the distinct matrix entries of the corresponding
matrices $(A,C)$. Let $\Phi_A$ denote the weights falling in the matrix
$A$, and let $\Phi_C$ denote those falling in $C$.
Let $M \subset \Phi_V$ be a subset, and define $V(M) = \{v \in V(\R) \mid v_a = 0 \, \forall a \in M\}$.
Let $m_a = \# M \cap \Phi_A$ and $m_c = \# M \cap \Phi_C$.
The volume of the projection to $V(M)(\Z)$ is of the order of
\[
d^{\dim V - \# M} \lambda^{2m_c - 2m_a} \prod_{w \in M} w(t)^{-1}
\]
We can compare the weights to the highest weights, which correspond to
the entries $a_{1,2n}$ and $c_{1,2n}$. Then, the volume of the projection
is
\[
\ll d^{\dim V} (d\lambda^2 t_1 t_{2n})^{-m_a} (d\lambda^{-2} (t_1 t_{2n})^{-1})^{-m_c}.
\]
This is at its biggest when $\{m_a,m_c\} = \{0,1\}$, in either order,
which corresponds to the projections in the case when $M = \{a_{1,2n}\}$
or $M = \{c_{1,2n}\}$. In the first of these cases (the second one is
analogous), we can compute that $a_{1,2n}(t) \delta^{-1}(t)$ can be written
as a product of $s_i$ with strictly positive exponents. Given that
$d^{-1} \ll \lambda^2 \ll d$, we get that
\[
\int_{d = 1}^X \int_{d^{-1/2} \ll \lambda \ll d^{1/2}} \int_{t_{i+1}/t_{i} \ll 1} d^{\dim V - 1} \lambda^{-2} (t_1t_{2n})^{-1} \, \delta^{-1}(t) \, d^{\times}\lambda \, d^{\times} t \,  d^{\times}d \ll X^{\dim V}.
\]
As for the main body, we compute that it is of the order of
\[
\int_{d = 1}^X \int_{\lambda = d^{-1/2}}^{d^{1/2}} d^{\dim V} d^{\times}\lambda d^{\times} d  \sim  X^{\dim V}\log X.
\]
This concludes the proof.
\end{proof}

The proof of Theorem \ref{theo:count2} (and therefore of Theorem \ref{theo:count1})
then follows from Proposition \ref{prop:mainBody} and the three items of Proposition \ref{prop:cusp3}.

\subsection{Cutting off the cusp}
\label{subs:cusp}

In this section, we prove the first item of Proposition \ref{prop:cusp3}. We start with some
reducibility conditions:
\begin{lemma}
\label{lemma:cusp}
Let $v = (A,C) \in V(\Q)$ and $1 \leq k \leq n$. Suppose that both
$A$ and $C$ contain a $k \times (2n-k)$ top-right block of entries equal
to zero. Then $\Delta(v) = 0$.
\end{lemma}
\begin{proof}
In this situation, the product $AC$ is a block lower-triangular matrix
whose diagonals are square matrices of size $k$, $2n-2k$ and $k$. In particular,
the characteristic polynomial of $AC$ factors as the product of the three
(two if $k = n$) characteristic polynomials of the blocks. It is an elementary
computation to show that the characteristic polynomials of the first and
the last block are the same, and therefore that the total invariant
polynomial has discriminant zero.
\end{proof}
\begin{lemma}
\label{lemma:cuspRed}
Let $v = (A,C) \in V(\Q)$. Suppose that the top-right  $n \times n$ block of 
either $A$ or $C$ is zero. Then $v$ is reducible.
\end{lemma}
\begin{proof}
Suppose, without loss of generality, that the top-right $n \times n$ block of $A$ is equal to zero,
meaning that there exists $g \in \GL_{2n}(\Q)$ such that $gAg^* = J_{2n}$.
This implies that there is an element in the $G(\Q)$-orbit of $v$ which
is of the form $(J_{2n},C')$, so $v$ is reducible.
\end{proof}

With these conditions in mind, we can carry out the cutting off the cusp.
We note, however, that there will be some notable differences with usual
cases in the literature, as now $G$ is not semisimple. We note that the
action of $D \times \Lambda \times T$ on $V$ gives every entry of the matrices
of $V$ a weight. If $(A,C) \in V$, then $d \in D$ scales every entry of
$A$ and $C$ by $d$, while $\lambda \in \Lambda$ scales entries of $A$ by $\lambda^2$
and entries of $C$ by $\lambda^{-2}$. 

We recall the change of variables for $1 \leq m \leq 2n$:
\[
t_m = \prod_{k=1}^{m-1} s_k^{-k}\prod_{k=m}^{2n-1} s_k^{2n-k}.
\]
The weight of each entry of $(A,C) \in V$ with respect to $T$ can be written
as a product $\prod_{k=1}^{2n-1} s_k^{a_k}$ for some integers $a_k$. Let $\Phi_V$
denote the set of weights on $V$, which consists of the matrix entries
$a_{i,j}$ and $c_{i,j}$, with $i+j \leq 2n+1$. We denote $\Phi_V = \Phi_A \sqcup \Phi_C$,
where $\Phi_A$ consists of the $a_{i,j}$ entries and $\Phi_C$ consists 
of the $c_{i,j}$ entries. We can define an ordering on $\Phi_V$:
if $a,b \in \Phi_V$, we will say that $a \leq b$ if the exponent of every
$s_k$ in the weight of $a$ is lower or equal to the corresponding exponent
of $s_k$ in $b$. Then, an explicit computation shows that the entries $a_{1,2n}$ and $c_{1,2n}$ are maximal
elements of $\Phi_V$. Visually, the weights increase as we go to the top
and to the right of these matrices.

Let $M_0,M_1$ be disjoint subsets of $\Phi_V$, and let $S(M_0,M_1)$ denote
the elements $v \in V(\Z)$ such that $v_a = 0$ if $a \in M_0$ and $v_a \neq 0$
if $a \in M_1$. Let $\cC$ denote the collection of non-empty subsets $M_0$ such that
if $a \in M_0$ and $b \geq a$, then $b \in M_0$. Given $M_0 \in \cC$,
denote $\lambda(M_0) = \{a \in \Phi_V \setminus M_0 \mid M_0 \cup \{a\} \in \cC\}$
(i.e. the subset of maximal elements of $M_0$).

Then, to prove the first item of Proposition \ref{prop:cusp3}, it suffices to show that for every
$M_0 \in \cC$, either $S(M_0,\lambda(M_0))^{irr} = \emptyset$ or
$N(S(M_0,\lambda(M_0)),X) = O(X^{\dim V})$. First, we note that if $a_{i,i} \in M_0$
for any $1 \leq i \leq n$ or if $a_{n,n+1} \in M_0$, then $S(M_0,\lambda(M_0))^{irr} = \emptyset$
by Lemmas \ref{lemma:cusp} and \ref{lemma:cuspRed}. Then, for all $i$, there exists $v_i \in \lambda(M_0) \cap \Phi_A$ such
that $w(a_{i,i}) \leq w(v_i)$ if $1 \leq i \leq n-1$ and $w(a_{n,n+1}) \leq w(v_i)$.
If $v \in \lambda(M_0) \cap \Phi_A$, then we must have that $d\lambda^2 w(v) \gg 1$.
Therefore, given that $\prod_{i=1}^{n-1} w(a_{i,i}) w(a_{n,n+1}) = t_{n}/t_{n+1} < c$,
we have that
\[
\prod_{i=1}^n (d \lambda^2 w(v_i)) \gg d^n\lambda^{2n} \gg  1,
\]
which means that $\lambda^2 \gg d^{-1}$. The analogous argument with
entries in $C$ yields $\lambda^2 \ll d$. 

Pick $M_0 \in \cC$ and write $m_a = \# (M_0 \cap \Phi_A)$, $m_c = \# (M_0 \cap \Phi_C)$.
Then, Davenport's lemma
ensures that the number of lattice points in $S(M_0,\lambda(M_0))$ 
with height at most $X$ is
\[
\ll \int_{d = 1}^X\int_{\lambda = d^{-1/2}}^{d^{1/2}} \int_{s_k < c \, \forall k} d^{\dim V - \#M_0} \lambda^{2(m_c-m_a)} \prod_{v \in M_0} w(v)^{-1} \delta^{-1}(s) \, d^{\times}d \, d^{\times}\lambda \, d^{\times} s.
\]
If all the exponents of $s_k$ are strictly positive in the integral, then
the value of the integral is $O(X^{\dim V - \#M_0 + |m_c-m_a|})$, or
$O(X^{\dim V - \#M_0} \log X)$ if $m_c-m_a = 0$. In any case, it is $O(X^{\dim V})$,
as wanted. When the exponents of $s_k$ are not necessarily positive, we
use a trick due to Bhargava, adapted here to our circumstances. We note
that if $a \in M_1$, then
\[
d\lambda^{2} w(a) \gg 1 \quad \text{or} \quad d\lambda^{-2} w(a) \gg 1
\]
according to whether $a \in \Phi_A$ or $a \in \Phi_C$, respectively.
Then, we can multiply by a product $\prod_{a \in M_1} (d \lambda^{\pm 2}w(a))^{p(a)}$
for some values $p(a) \gg 1$ to get a new estimate.
\begin{prop}
\label{prop:cuspFunc}
Suppose there exists a function $p \colon M_1 \to \R_{\geq 0}$
such that:
\begin{itemize}
\item $\sum_{a \in M_1} p(a) \leq \# M_0$,
\item $\sum_{a \in M_1\cap \Phi_A} p(a) - \sum_{a \in M_1 \cap \Phi_C} p(a) = m_a - m_c$,
\item The exponents of $s_k$ in
\[
\prod_{a \in M_1}(w(a))^{p(a)} \prod_{a \in M_0} w(a)^{-1} \delta^{-1}(s)
\]
are all strictly positive.
\end{itemize}
Then $N(M_0,\lambda(M_0),X) = O(X^{\dim V})$.
\end{prop}
\begin{proof}
The conditions guarantee that $N(M_0,\lambda(M_0),X) = O(X^{\dim V-\#M_0 +\sum_{a \in M_1}p(a)}\log X)$,
so if $\sum_{a \in M_1} p(a) < \# M_0$ we are done. Otherwise, if $\sum_{a \in M_1} p(a) = \# M_0$,
replacing $p$ by $(1-\eps)p$ for a sufficiently small $\eps > 0$ yields 
\[
\begin{cases}
O(X^{\dim V - \eps \sum_{a \in M_1}p(a) + \eps |\sum_{a \in M_1 \cap \Phi_A}p(a) - \sum_{a \in M_1 \cap \Phi_C}p(a)|}) & \text{ if } \sum_{a \in M_1 \cap \Phi_A}p(a) \neq \sum_{a \in M_1 \cap \Phi_C}p(a) \\
O(X^{\dim V - \eps \sum_{a \in M_1}p(a)} \log X) & \text{ otherwise}.
\end{cases}
\]
In any case, it is $O(X^{\dim V})$.
\end{proof}
\begin{prop}
For every $M_0 \in \cC$, there exists $p \colon M_1 \to \R_{\geq 0}$ satisfying the conditions of
Proposition \ref{prop:cuspFunc}.
\end{prop}
\begin{proof}
We proceed by induction on $n \geq 2$. We start by doing the $n = 2$ case
explicitly. We label the weights of $\Phi_V$ in this case:
\[
(A,C) = \lp \begin{pmatrix} 6 & 4 & 2 & 1 \\ 8 & 5 & 3 & \\ 9 & 7 & & \\ 10 & & & \end{pmatrix}, \begin{pmatrix} 16 & 14 & 12 & 11 \\ 18 & 15 & 13 & \\ 19 & 17 & & \\ 20 & & & \end{pmatrix} \rp
\]
For reference, we also write explicitly the action of $T$ on $(A,C)$,
which is indexed by $(s_1,s_2,s_3)$:
\[
\lp \begin{pmatrix} (2,0,-2) & (2,0,2) & (2,4,2) & (6,4,2) \\ (-2,0,-2) & (-2,0,2) & (-2,4,2) & \\ (-2,-4,-2) & (-2,-4,2) & & \\ (-2,-4,-6) & & & \end{pmatrix}, \begin{pmatrix} (-2,0,2) & (2,0,2) & (2,4,2) & (2,4,6) \\ (-2,0,-2) & (2,0,-2) & (2,4,-2) & \\ (-2,-4,-2) & (2,-4,-2) & & \\ (-6,-4,-2) & & & \end{pmatrix} \rp
\]
We have that $\delta^{-1}(s) = s_1^{12} s_2^{16} s_3^{12}$. Taking
into account Lemmas \ref{lemma:cusp} and \ref{lemma:cuspRed}, we write down the list of all non-trivial possibilities for $M_0$,
together with a choice of $p \colon M_1 \to \R_{\geq 0}$. We omit repetitions
we would get by swapping $A$ and $C$. Here, $\eps$ denotes a small positive
real number.
\[
\begin{array}{c|c|l}
M_0 & \text{weight} & p \\
\hline
\{1 \} & d^{19} \lambda^{-2} (6,12,14) & 1 \cdot (2) \\
\{1,11\} & d^{18} (4,8,4) & 0 \\
\{1,2\} & d^{18} \lambda^{-4} (4,8,8) & 2 \cdot (4) \\
\{1,2,11\} & d^{17} \lambda^{-2} (2,4,2) & 1 \cdot (4) \\
\{1,2,11,12\} & d^{16} (0,0,0) & \eps (3) + \eps(13) + \eps(4)+ \eps(14) \\
\{1,2,4\} & d^{17} \lambda^{-6} (2,8,6) & 1 \cdot \frac{3}{2}(6)+\frac{3}{2}(5) \\
\{1,2,4,11\} & d^{16} \lambda^{-4} (0,4,0) & (1+\eps)(6)+(1+\eps)(5)+2\eps(14) \\
\{1,2,4,11,12\} & d^{15} \lambda^{-2} (-2,0,-2) & \eps(3)+\eps(13)+(2+\frac{\eps}{2})(6)+(2+\frac{\eps}{2})(5) + (1+\eps)(14) \\
\end{array}
\]
Now, by induction we assume that the theorem holds for $n-1$. Let $M_0 \in \cC$.
Inside $\Phi_V$, let $\Phi_{n-1}$ denote the subset of weights corresponding
to the inner $(2n-2) \times (2n-2)$ matrices of $(A,C)$ (that is, removing
the first and last rows and columns). If $M_0 \cap \Phi_{n-1}$ is non-empty,
then use induction to get a function $p_{n-1} \colon M_1 \cap \Phi_{n-1} \to \R_{\geq 0}$
satisfying:
\begin{itemize}
\item $\sum_{a \in M_1 \cap \Phi_{n-1}} p_{n-1}(a) \leq \#(M_0 \cap \Phi_{n-1})$
\item $\sum_{a \in M_1 \cap \Phi_{n-1} \cap \Phi_A} p_{n-1}(a) - \sum_{c \in M_1 \cap \Phi_{n-1} \cap \Phi_C} p_{n-1}(c) =  \#(M_0 \cap \Phi_A \cap \Phi_{n-1}) - \#(M_0 \cap \Phi_C \cap \Phi_{n-1})$
\item The exponents in 
\[
\prod_{a \in M_1 \cap \Phi_{n-1}}(w(a))^{p_{n-1}(a)} \prod_{a \in M_0 \cap \Phi_{n-1}} w(a)^{-1} \prod_{k=2}^{2n-2}s_k^{2n(k-1)(2n-k-1)}
\]
are strictly positive.
\end{itemize}
If $M_0 \cap \Phi_{n-1}$ is empty, we can carry out the argument with $p_{n-1} = 0$.
Then, it suffices to find $p_1 \colon M_1 \to \R_{\geq 0}$ satisfying:
\begin{itemize}
\item $\sum_{a \in M_1} p_1(a) \leq \#(M_0 \cap (\Phi_V \setminus \Phi_{n-1}))$
\item $\sum_{a \in M_1 \cap \Phi_A} p_1(a) - \sum_{c \in M_1 \cap \Phi_C} p_1(c) =  \#(M_0 \cap (\Phi_A \setminus \Phi_{n-1})) - \#(M_0 \cap (\Phi_C \setminus \Phi_{n-1}))$
\item The exponents in
\[
\prod_{a \in M_1}(w(a))^{p_{1}(a)} \prod_{a \in M_0 \cap (\Phi_V \setminus \Phi_{n-1})} w(a)^{-1} \prod_{k=1}^{2n-1}s_k^{2n(2n-1)}
\]
are non-negative.
\end{itemize}
Let $m_a = \#((M_0 \cap \Phi_A) \setminus \Phi_{n-1})$ and $m_c = \#((M_0 \cap \Phi_C) \setminus \Phi_{n-1})$,
which are the number of elements of $M_0$ in the top row of $A$ and $C$
respectively. Then, $a_{1,2n-m_a},c_{1,2n-m_c} \in M_1$. We select $p_1$
according to the following cases, where without loss of generality we
assume that $a \geq c$:
\begin{itemize}
\item $m_a+m_c < 2n$. In this situation, we choose the function $p_1$
to be $p_1(a_{1,2n-m_a}) = m_a$, $p_1(c_{1,2n-m_c}) = m_c$ and $0$ everywhere else.
\item $m_a+m_c \geq 2n$ and $m_c \leq n$. In this situation, we choose
$p_1(a_{1,2n-m_a}) = m_a-1$, $p_1(a_{n,n+1}) = 1$, $p(c_{1,2n-m_c}) = m_c$ and $0$ everywhere else.
\item $m_a+m_c \geq 2n$ and $m_c > n$. Let $M = m_a+m_c-2n+\frac{1}{2}$,
where we note that $M \leq m_a$ and $M \leq m_c$.
We note that, for all $1 \leq k \leq n-1$, at least one of $a_{k,k+1}$
and $c_{k,k+1}$ belongs to $M_1$ by Lemma \ref{lemma:cusp}. If $a_{k,k+1} \in M_1$,
then we note that $w(a_{k,k+1})w(c_{k,k}) = s_{2n-k}^{2n}$ and $w(a_{k,k+1})w(c_{k+1,k+1}) = s_{k}^{2n}$;
with similar formulae if $c_{k,k+1} \in M_1$.
We also note that $a_{n,n+1},c_{n,n+1} \in M_1$ by Lemma \ref{lemma:cuspRed},
and that $w(a_{n,n+1})w(c_{n,n+1}) = s_n^{4n}$.
With all that said, we choose $p_1$ in the following manner:
\begin{itemize}
\item $p_1(a_{1,2n-m_a}) = m_a-M$ and $p_1(c_{1,2n-m_c}) = m_c-M$.
\item For all $2n-m_c \leq k \leq n-1$, if $a_{k,k+1} \in M_1$ we put
$p_1(a_{k,k+1}) = 2$, $p_1(c_{k,k}) = 1$ and $p_1(c_{k+1,k+1}) = 1$.
Otherwise, we put $p_1(c_{k,k+1}) = 2$, $p_1(a_{k,k}) = 1$ and $p_1(a_{k+1,k+1}) = 1$.
\item For all $2n-m_a \leq k < 2n-m_c$, if $a_{k,k+1} \in M_1$ we put
$p_1(a_{k,k+1}) = 1$ and $p_1(c_{k+1,k+1}) = 1$. Otherwise, we put
$p_1(c_{k,k+1}) = 1$ and $p_1(a_{k,k}) = 1$.
\item $p_1(a_{n,n+1}) = p_1(c_{n,n+1}) = \frac{1}{2}$.
\end{itemize}
\end{itemize}

\end{proof}

\subsection{Proof of Theorem \ref{theo:mainPsi}}

We are now in a position to prove Theorem \ref{theo:mainPsi}. As with
Theorem \ref{theo:mainPhiCong}, the result holds even if we apply finitely
many congruence conditions on $B(\Z)$.

\begin{theorem}
\label{theo:PsiCong}
Let $\cB \subset B(\Z)$ be a subset defined by finitely many congruence 
conditions. Then
\[
\frac{\sum_{b \in \cB, \, \Ht(b) < X} \# \Sel_{\psi^{\vee}}(A_b)}{\#\{b \in \cB \mid \Ht(b) < X\}} \ll \log X.
\]
The implied constant depends on $\cB$.
\end{theorem}
\begin{proof}
Let us start with the case $\cB = B(\Z)$.
In this case, the bound follows from Theorem \ref{theo:count1}: we note
that each element of $\Sel_{\psi^{\vee}} A_b$ embeds inside a different
$G_C(\Z)$-orbit of $V_{C,b}(\Z) \cap V^{sol}(\R)$. If we restrict to $\Ht(b) < X$,
then the number of such orbits is $O(X^{\dim V}\log X)$, since:
\begin{itemize}
\item The irreducible $G_C(\Z)$-orbits in $V_C(\Z)$ are controlled by Theorem \ref{theo:count1}
and are $O(X^{\dim V}\log X)$; and
\item The reducible $G_C(\Z)$-orbits in $V_C(\Z)$ are, at most, twice
the number of orbits in $G_A(\Z)$ acting on $V_A(\Z)$, so they are $O(X^{\dim V})$.
\end{itemize}
Hence the theorem holds for the case $\cB = B(\Z)$. When we apply finitely
many congruence conditions, the proof carries over with minimal changes
by combining Theorem \ref{theo:count1} with the methods of Section \ref{section:mainProof}.
\end{proof}

\section{Tamagawa ratios}
\label{section:tamagawa}

We will prove Theorem \ref{theo:mainMoments} by using the machinery developed in the recent preprint \cite{KS}.
We will look at the Tamagawa ratio
\[
\cT(A_b/J_b) = \frac{\# \Sel_{\phi_M^{\vee}}(A_b^{\vee})}{\# \Sel_{\phi_M}(J_b)}.
\]
In our situation, the Greenberg--Wiles formula \cite[Theorem 8.7.9]{NSW}
states that
\[
\cT(A_b/J_b) = \prod_{p \leq \infty} c_p(\phi_{M,b}^{\vee}) = \prod_{p \leq \infty} \frac{\# J_b(\Q_{p})/\phi_M^{\vee}(A_b^{\vee}(\Q_p))}{2},
\]
This infinite product is convergent, as for all primes away from $\{2,\infty\}$
of good reduction for $J_b$ the local factor is equal to $1$. In fact,
we also have that $\# J_b(\Q_{p})/\phi_M^{\vee}(A_b^{\vee}(\Q_p)) \leq 8$,
as it is a subset of $H^1(\Q_p,\Z/2\Z) \cong \Q_p^{\times}/(\Q_p^{\times})^2$
(so in particular, $c_p(\phi_{M,b}^{\vee}) \leq 2$ if $p \neq 2$, and
also $c_p(\phi_{M,b}^{\vee}) \geq \frac{1}{2}$). By \cite[Lemma 3.8]{Sch}, if
$p \neq 2,\infty$ we additionally have
\[
c_p(\phi_{M,b}) = \frac{c_p(J_b)}{c_p(A_b^{\vee})} = \frac{\# J_b(\Q_p)/J_{b,0}(\Q_p)}{\# A_b^{\vee}(\Q_p)/A_{b,0}^{\vee}(\Q_p)}.
\]
We note that a consequence of \cite[Lemma 3.8]{Sch}, combined with Lemma
\ref{lemma:selmerWeights} is that the Tamagawa number of $A_b$ at a prime
$p \neq 2,\infty$ coincides with the Tamagawa number of $A_b^{\vee}$ at
$p$.

A consequence of the above discussion is that we can obtain a lot of
information about the sizes of Selmer groups if we understand the ratios
between the Tamagawa numbers of $J_b$ and $A_b$. For statistical purposes,
we can ignore the cases $p = 2,\infty$, as they are uniformly controlled.
We are interested in the cases when $p$ is of bad reduction for $J_b$,
which for $C_b \colon y^2 = xf(x)$ happens exactly when $p \mid \Delta(xf(x)) = f(0)^2\Delta(f)$.
It is sufficient for us to understand the Tamagawa numbers when the
discriminant is ``as squarefree as possible''. For $J_b$, the Tamagawa
numbers can be computed with relatively standard arguments:
\begin{lemma}
\label{lemma:Tamagawa0}
Suppose $\Delta(b) \neq 0$ and $p \neq 2,\infty$. We have that
\[
c_p(J_b) = \begin{cases}
1 & \text{ if } p \nmid f(0) \text{ and } p \parallel \Delta(f), \\
2 & \text{ if } p \parallel f(0) \text{ and } p \nmid \Delta(f).
\end{cases}
\]
\end{lemma}
\begin{proof}
In the first case, the discriminant of $C_b$ is squarefree, and it is
well-known that in that case the Tamagawa number is $1$. To elaborate: the 
equations for the standard affine chart of $C_b$ already define a minimal
regular model, and the special fibre $\cC_{b,\F_p}$ has only has one 
connected component. Then, because the reduction is semistable, the
Tamagawa number can be computed from the dual graph of $\cC_{b,\F_p}$,
and because this has only one component it follows that $c_p(J_b) = 1$.

In the second case, the minimal regular model for $C_b$ is obtained by
performing one blow-up at the point $(0,0) \in C_b(\Q_p)$. The result
is a minimal regular model $\cC_{b,\Z_p}$ such that $\cC_{b,\F_p}$ has
two connected components, intersecting at two different points. It follows
from \cite[Theorem 9.6.1]{BLR} that the Tamagawa number is $2$ in that case.
\end{proof}
Obtaining the Tamagawa numbers for $A_b$ is a bit more subtle. We start
by fixing some notation: if $A$ is an abelian variety
over $\Q_p$ with semistable reduction, then the connected component of the 
special fibre of its Néron model $\cA$ fits into an exact sequence
\[
\begin{tikzcd}
1 \arrow[r] & \cT \arrow[r] & \cA_{\F_p}^0 \arrow[r]  & \cB \arrow[r] & 1,
\end{tikzcd}
\]
where $\cB$ is an $\F_p$-abelian variety, and $\cT$ is a torus. Let $\ell$
be a prime, and denote the Tate module of $A$ by $T_{\ell}(A)$. There is
a canonical filtration
\[
T_{\ell}(A) \supset \cM_f(A) \supset \cM_t(A) \supset 0,
\]
where $\cM_f(A) = T_{\ell}(A)^{I_p}$ (here $I_p$ denotes the inertia
subgroup of $\Gal(\ol{\Q_p}/\Q_p)$), and $\cM_t(A)$ is the orthogonal complement of
$\cM_f(\hat{A})$ with respect to the Weil pairing.
We then have the following result (cf. \cite[Lemma 3]{BK}):

\begin{lemma}
Let $B,B'$ be abelian varieties over $\Q_p$, and let $\kappa$ be a $G_{\Q_p}$-submodule
of $B[\ell]$, where $\ell$ is a prime. Suppose $\varphi \colon B \to B'$
is an isogeny with kernel $\kappa$. Denote by $\ol{\cM_f(A)},\ol{\cM_t(A)}$
the projections of $\cM_f(A),\cM_t(A)$ to $B[\ell]$. Then
\[
\ord_{\ell}(\Phi_{A'^{\vee}}(\ol{\F_p})) - \ord_{\ell}(\Phi_{A^{\vee}}(\ol{\F_p})) = 
\dim(\kappa \cap \ol{\cM_f(A)}) + \dim(\kappa \cap \ol{\cM_t(A)}) - \dim(\kappa).
\]
\end{lemma}
\begin{lemma}
\label{lemma:tamagawa}
Suppose $\Delta(b) \neq 0$  and $p \neq 2,\infty$. We have that
\[
c_p(A_b) = c_p(A_b^{\vee}) = \begin{cases}
1 & \text{ if } p \nmid f(0) \text{ and } p \parallel \Delta(f), \\
1 & \text{ if } p \parallel f(0) \text{ and } p \nmid \Delta(f).
\end{cases}
\]
\end{lemma}
\begin{proof}
In both cases, it suffices to understand how the kernel $M$ of the $2$-isogeny
$\phi_M \colon J_b \to A_b$ interacts with $\ol{\cM_f(A)}$ and $\ol{\cM_t(A)}$.

In the first case, the statement is equivalent to asking that 
$M \subset \ol{\cM_f(A)}$ but $M \cap \ol{\cM_t(A)} = \{0\}$.
The former is forced by the fact that $\# \Phi_{A_b}(\ol{\F_p}) \geq 1$.
Let $\tilde{C}$ be the normalisation of the special fibre of $C$, which
is explicitly given as follows: if $xf(x) \equiv (x-a)^2g(x) \pmod{p}$
for some $a \neq 0$ in $\F_p$, then $\tilde{C} \colon t^2 = g(x)$. Let
$\tilde{\cJ}$ be the Jacobian of $\tilde{C}$. Then we have
\[
1 \to T \xhookrightarrow{} \cJ_{\F_p}^0 \to \tilde{\cJ} \to 1.
\]
Suppose that the roots of $xf(x)$ are $x_0 = 0, x_1,\dots,x_{2n}$,
with the corresponding $P_i = (x_i,0) \in C(\ol{\F_p})$, and suppose that $x_{2n-1} = x_{2n} \pmod{p}$.
Then any element of $\cJ_{\F_p}^0[2]$ can be uniquely written as $\sum_{i \in I} [(P_i)-\infty]$
for some subset $I \subset \{0,\dots,2n-2\}$.
In this situation, $\tilde{\cJ}[2]$ is the quotient of $\cJ_{\F_p}^0[2]$
by the relation $[(P_0)-\infty] + \dots + [(P_{2n-2})-\infty] = 0$, and
$T[2]$ is therefore generated by $[(P_0)-\infty] + \dots + [(P_{2n-2})-\infty]$,
which may be identified with $[(P_{2n-1})-\infty] + [(P_{2n})-\infty]$.
We note in particular that $[(P_0)-\infty] \notin T[2]$, and given that
$\cM_t(A) \cong T_{\ell}(T)$, we conclude that $M \cap \ol{\cM_t(A)} = \{0\}$.

In the second case, it suffices to see that $M \cap \ol{\cM_f(A)} = \{0\}$, or equivalently
that $[(0,0)-\infty] \notin \ol{\cM_f(A)}$. This follows from observing that $(0,0)$ and $\infty$ lie in different
components of $\cC$, and using explicit descriptions of Raynaud's specialisation
maps as in \cite[Appendix A]{Baker}.
\end{proof}
\begin{obs}
Note that the analogue of Lemma \ref{lemma:tamagawa} would not be true
in the genus $1$ case. There we would have that $c_p(A_b) = 2$ if
$p \nmid f(0)$ and $p \parallel \Delta(f)$, as in that situation $T = \cJ_{\F_p}^0$.
\end{obs}

Therefore, in ``most'' cases of bad reduction, we have that $c_p(J_b)/c_p(A)$
is either $1$ or $2$. This imbalance forces the overall Selmer ratio to
be large on average.

\begin{theorem}
\label{theo:KStamagawa}
Let $\kappa\geq 0$ be a real number. As $X \to \infty$, we have that
\[
\frac{\sum_{\Ht(b) < X} (\cT(A_b/J_b))^{\kappa}}{\sum_{\Ht(b) < X} 1} \asymp_{\kappa} (\log X)^{2^{\kappa}-1}.
\]
The implied constant depends on $\kappa$.
\end{theorem}
\begin{proof}
This is a direct consequence of \cite[Proposition 6.16]{KS}, using the
Tamagawa number computations of Lemmas \ref{lemma:Tamagawa0} and \ref{lemma:tamagawa}.
Let us write out some details.

The result \cite[Proposition 6.16]{KS} asserts that
\[
\frac{\sum_{\Ht(b) < X} (\cT(A_b/J_b))^{\kappa}}{\sum_{\Ht(b) < X} 1} \asymp_{\kappa} (\log X)^{\beta(\kappa)},
\]
for some function $\beta \colon \R_{\geq 0} \to \R_{\geq 0}$
\[
\beta(\kappa) = \max \left\{ 0,\sum_{t} \gamma(t)(t^{\kappa}-1) \right\},
\]
where in our situation we sum over rational numbers $t = a/b$ where
$a,b$ are coprime integers that divide $2$. The function $\gamma(t)$
is defined before the statement of \cite[Proposition 6.16]{KS}, and
simplifies quite substantially in our setting (in particular, in the
notation of Koymans--Smith, we can take $L_t = \Q$ for all $t$). In this
case, $\gamma(t)$ is defined so that the density of elements 
$\ol{b} \in B(\Z/p^2\Z)$ such that $\cT(A_b/J_b) = t$
for all $b \in B(\Z)$ reducing to $\ol{b}$ modulo $p^2$ is equal to
$\gamma(t)p^{-1} + O(p^{-3/2})$ for all sufficiently large primes $p$.

Using Lemmas \ref{lemma:Tamagawa0} and \ref{lemma:tamagawa}, we compute
$\gamma(2) = \gamma(1) = 1$ and $\gamma(t) = 0$ for any other $t$, so
\[
\beta(\kappa) = 2^{\kappa} - 1,
\]
as claimed.
\end{proof}

With this, we can prove Theorem \ref{theo:mainMoments}: we recall the statement.
\begin{theorem}[= Theorem \ref{theo:mainMoments}]
\label{theo:SelMoments}
Let $\kappa \geq 0$. As $X \to \infty$, we have that
\[
\frac{\sum_{\Ht(b) < X} (\# \Sel_{\phi_M^{\vee}} A_b^{\vee})^{\kappa} }{\sum_{\Ht(b) < X} 1} \asymp_{\kappa} (\log X)^{2^{\kappa}-1},
\]
\[
\frac{\sum_{\Ht(b) < X} (\# \Sel_{\phi_M}J_b)^{\kappa}}{\sum_{\Ht(b) < X} 1} \asymp_{\kappa} 1
\]
and
\[
\frac{\sum_{\Ht(b) < X} (\# \Sel_{\psi^{\vee}} A_b)^{\kappa} }{\sum_{\Ht(b) < X} 1} \gg_{\kappa} (\log X)^{2^{\kappa}-1}.
\]
All the implied constants depend on $\kappa$.
\end{theorem}
\begin{proof}
The first and third lower bounds follow immediately from the fact that 
\[
\# \Sel_{\phi_M^{\vee}} A_b^{\vee} \geq \cT(A_b/J_b) \quad \text{and} \quad \# \Sel_{\psi^{\vee}} A_b^{\vee} \geq \frac{\cT(A_b/J_b)}{4},
\]
combined with Theorem \ref{theo:KStamagawa}.  The upper bound for the moments of $\# \Sel_{\phi_M^{\vee}} A_b$ follow from
\cite[Theorem 6.4]{KS}, which asserts that
\[
\frac{\sum_{\Ht(b) < X} (\# \Sel_{\phi_M^{\vee}} A_b^{\vee})^{\kappa} }{\sum_{\Ht(b) < X} 1} \asymp_{\kappa} \frac{\sum_{\Ht(b) < X} (\cT(A_b/J_b))^{\kappa} }{\sum_{\Ht(b) < X} 1}.
\]
For the moments of 
$\# (\Sel_{\phi_M}(J_b))^{\kappa}$, \cite[Theorem 6.4]{KS}
gives
\[
\frac{\sum_{\Ht(b) < X} \# (\Sel_{\phi_M}(J_b))^{\kappa}}{\sum_{\Ht(b) < X} 1} \asymp_{\kappa} \frac{\sum_{\Ht(b) < X} (\cT(A_b/J_b))^{-\kappa} }{\sum_{\Ht(b) < X} 1},
\]
for which the same proof as in Theorem \ref{theo:KStamagawa} gives
\[
\frac{\sum_{\Ht(b) < X} (\cT(A_b/J_b))^{-\kappa} }{\sum_{\Ht(b) < X} 1} \asymp_{\kappa} (\log X)^{\beta(\kappa)},
\]
for $\beta(\kappa) = \max\{0,2^{-\kappa}-1\}$. As $\kappa \geq 0$,
we have that $\beta(\kappa) = 0$, proving the last claim.
\end{proof}

\begin{obs}
As with the other main results, Theorem \ref{theo:mainMoments} also holds
when we apply finitely many congruence conditions on $B(\Z)$, since the
Tamagawa ratio computations matter only for sufficiently large $p$ and
the results in \cite{KS} still hold under these congruence conditions.
\end{obs}

\begin{obs}
Note that this last proof also gives the trivial bound
\[
\frac{\sum_{\Ht(b) < X} (\# \Sel_{\psi} J_b)^{\kappa} }{\sum_{\Ht(b) < X} 1} \gg_{\kappa} 1,
\]
and it is reasonable to expect (cf. \cite[Conjectures 1 and 2]{KS})
that this is the correct order of growth (i.e. that all the
moments of $\# \Sel_{\psi} J_b)^{\kappa}$ are bounded), but we do not
have the methods to prove this. By using purely geometry-of-numbers methods,
Theorem \ref{theo:integral} would suggest that we look at the representation $(G_B,V_B)$. Some standard
computations show that the number of irreducible $G(\Z)$-orbits in $V(\Z)$
is of the order of $X^{\dim V_B}$, but that the number of integral orbits
in the cusp is $X^{\dim V_B}\log X$. However, the infinitely many congruence
conditions corresponding to the Selmer elements appear to have density
zero, so normal sieving arguments would not work. It appears that any
precise results on this direction within this framework would require a very precise control
over error terms; similarly to \cite[Conjecture 8.19]{LagaThesis}, this
appears to be rather complicated in general with our current methods.
\end{obs}

\printbibliography
\end{document}